\documentclass[a4paper,12pt]{article}
\usepackage{amsfonts}
\usepackage{setspace}
\usepackage{wrapfig}
\usepackage{subfig}
\usepackage{graphicx}    
\usepackage[margin=1in]{geometry}
\usepackage{natbib}
\usepackage{bbm}
\usepackage{float}
\usepackage[toc,page]{appendix}
\usepackage{color}

\usepackage{mathtools}
\usepackage{amsmath}

\newtheorem{theorem}{Theorem}

\newtheorem{lemma}[theorem]{Lemma}

\newenvironment{proof}[1][Proof]{\noindent\textbf{#1.} }{\ \rule{0.5em}{0.5em}}
\input{tcilatex}
\allowdisplaybreaks
\begin{document}

\title{Semi-analytical solution of a McKean-Vlasov equation with feedback through hitting a boundary}
\author{
Alexander Lipton\footnote{Massachusetts Institute of Technology, Connection Science, Cambridge, MA, USA and Ecole Polytechnique Federale de Lausanne, Switzerland, E-mail: alexlipt@mit.edu}, 
Vadim Kaushansky\footnote{Mathematical Institute \& Oxford-Man Institute, University of Oxford, UK, E-mail: vadim.kaushansky@maths.ox.ac.uk}\thanks{The second author gratefully acknowledges support from the Economic and Social Research Council and Bank of America Merrill Lynch},
Christoph Reisinger\footnote{Mathematical Institute  \& Oxford-Man Institute, University of Oxford, Andrew Wiles Building, Woodstock Road, Oxford, OX2 6GG, UK, E-mail: christoph.reisinger@maths.ox.ac.uk}}    
\date{}

\maketitle

\begin{abstract}
In this paper, we study the non-linear diffusion equation associated with a particle system where 
the common drift depends on the rate of absorption of particles at a boundary.
We provide an interpretation as a structural credit risk model with default contagion in a large interconnected banking system.
Using the method of heat potentials, we derive a coupled system of Volterra integral equations for the transition density and for the loss through absorption.
An approximation by expansion is given for a small interaction parameter. We also present a numerical solution algorithm and conduct computational tests.
\end{abstract}

\section{Introduction}

In this paper, we derive semi-analytical solutions for the density of interacting particles where the interaction results from shocks to the system when particles hit a boundary.
Equations of this type have arisen recently as models for the ``integrate and fire'' behaviour of neuronal networks and for systemic default  risk in networks of interconnected banks.

Structural default models, where a bank's default is triggered by its assets falling below its liabilities, have been studied for decades since the seminal work of \cite{Merton}. There are several limitations to the basic version of these models:
most do not take into account that banks are interconnected, as a result, ignoring the possibility of contagious defaults (but see, e.g., \cite{haworth2006modeling, HaworthReisinger}).
To address this, \cite{Lipton2015} combined the structural  and \cite{Eisenberg} framework to consider not only external liabilities, but also mutual liabilities.

A further problem is the curse of dimensionality. Numerical and analytical PDE techniques are typically applied up to dimension three (\cite{LiptonItkin2015}, \cite{LiptonItkin2015b}, \cite{kaushansky2018numerical}, \cite{kaushansky2018transition}); for any larger dimension, only Monte Carlo methods are usually considered viable, which are slow to converge and noisy by nature. 

When the banking system is large and homogenous, and only macroscopic quantities are of interest,
one can consider a large pool approximation of the banks' asset value processes (technically, by taking the limit of their empirical measure for an infinite number of banks).
This approach was first studied by \cite{bush2011}. Following on, \cite{Shkolnikov} and \cite{Andreas} took into account interaction effects by considering a particle system with positive feedback from the firms' defaults. This leads to McKean--Vlasov type equations, which model a typical representative of the banking system whose dynamics depends on the losses in the wider system. An alternative viewpoint is provided by the \cite{Lipton2015} model when taking the number of banks there to infinity. This route leads to the same equation, as shown by our derivation in the next section.

\cite{Andreas} assumed zero correlation between banks with linear dependence of the interaction on the loss function and obtained  existence and uniqueness of the solution, with the necessity of blow-ups for large enough interaction parameter. \cite{hambly2018spde} and \cite{ledger2018mercy} introduced positive correlation between banks, while \cite{Shkolnikov} and \cite{nadtochiy2018mean} considered a nonlinear dependence through the loss function.
Very recently, \cite{ichiba2018dynamic} derived a McKean-Vlasov SDE in nonlinear jump-diffusion form 
for  the  average  bank  reserves in an interacting banking  system with  local  and  mean-field  default  intensities.

Earlier, a model similar to the one studied here was found in neuroscience, where a large network of electrically coupled neurons can be described by McKean--Vlasov type equations (\cite{caceres2011analysis, carrillo2013classical,  delarue2015particle, delarue2015global}). If a neuron's potential reaches some fixed threshold, it jumps to a higher potential level and sends a signal to other neurons.

Originally, McKean--Vlasov type equations were suggested by \cite{kac1956foundations} as a stochastic toy model for the Vlasov kinetic equation of plasma, with  a detailed study by \cite{mckean1966class}. In recent years, mean-field problems, and McKean--Vlasov type equations in particular, have become a very popular topic in applied mathematics from both a theoretical and practical perspective. Different versions of such problems, apart from the specific form in the papers above, have been applied to mathematical finance, e.g., in portfolio optimization (\cite{borkar2010mckean} consider optimal allocation into sectors for a large number of stocks) and in game theory (e.g., \cite{huang2006large} discuss an agent's optimal  behavior with respect to a mass effect).

From a numerical prospective, there are established simulation methods for typical McKean--Vlasov equations (see \cite{bossy1997stochastic} and \cite{antonelli2002rate})
and, more recently, several authors have analyzed multilevel and multi-index schemes (see \cite{ricketson2015multilevel}, \cite{szpruch2017iterative}, \cite{haji2018multilevel})
and importance sampling (\cite{reis2018importance}).

However, none of these methods cover the models described above due to the singular, path-dependent nature of the feedback.
Here, we consider the \cite{Andreas} version for simplicity.
For this model, \cite{kaushansky2018simulation} proposed an Euler-type particle method and proved convergence with order $1/2$ in the timestep,
which can be improved to $1$ using Brownian bridges.
In this paper, we show how to solve these equations by reformulation first as a non-linear free boundary problem similar to the classical Stefan problem
and then as a system of two coupled Volterra equations.

First, for \emph{given}  drift term from the mean-field interaction,
the problem is formulated as diffusion problem on a semi-infinite domain with curvilinear boundary, and its solution represented semi-analytically by the method of heat potentials.
A detailed introduction to heat potentials can be found in \cite{watson2012introduction} and 
\cite{tikhonov2013equations} (pp.~530--535).
The first use of the method of heat potentials in mathematical finance is found in \cite{Lipton2001} for pricing path-dependent options with curvilinear barrier
(Section 12.2.3, pp.~462--467).

Second, expressing the interaction term by the solution from the first step results in two coupled Volterra equations.
For early applications of heat potentials to versions of the Stefan problem see already, e.g., \cite[Part II, Chapter 1]{rubinstein1971stefan}.
These singular Volterra equations are then solved by either an expansion method for small interaction parameter, or numerically by discretisation and Newton-Raphson iteration. 

An expansion approach for a certain McKean--Vlasov equation with mean-field interaction through the drift was recently studied in \cite{gobet2018analytical},
who perform an iterative two-step procedure which decouples the nonlinearity arising from the dependence of the drift on the law of the process from the standard dependence on the state-variables. 
The present paper differs not only in the solution approach taken, but fundamentally in the considered mean-field interaction (through hitting times rather than the expectation of the drift) and the parameter of the expansion (for small drift interaction rather than small volatility).

To assess the accuracy of the (first order) perturbation solution and to illustrate the behaviour for strong interactions where the expansion breaks down,
we describe a simple numerical algorithm, but refer the reader to the large and well-established body of literature on more advanced numerical methods for Volterra equations (see Section \ref{numerical_solution}).

The rest of the paper is organized as follows: in Section \ref{sec:limit},
we provide an alternative derivation of the model described in \cite{Andreas} by taking the limiting case for infinitely many firms in the model of \cite{Lipton2015}; 
in Section \ref{sec:HP}, we derive a solution for the first passage density of Brownian motion over a curve in terms of a Volterra equation, using the method of heat potentials,
and
thus obtain the interaction term in the original McKean--Vlasov equation;
in Section \ref{sec:solMcKV}, we we consider a perturbation method and a numerical method for the corresponding system of Volterra equations;
in Section \ref{sec:results}, we show numerical illustrations and compare the methods; in Section \ref{sec:conclusion} we conclude.

\section{Mean-field limit for large banking system}
\label{sec:limit}

Following \cite{Lipton2015}, we consider a system of $N$ banks with external as well as mutual assets and liabilities.
We denote by
$L_i$ the external liabilities for bank $i$ and by
$L_{ij}$ the liability from bank $i$ to bank $j$.

Assume that the dynamics of bank $i$'s total external assets is governed by
\begin{equation*}
	\frac{d  A^i_t}{A^i_t} =  \mu_i \, dt  + \sigma_i  \, d W_t^i, 
\end{equation*}
where $W^i$ are independent standard Brownian motions for $1 \le i \le n$,
and the liabilities, both external $L_i$ and mutual $L_{ij}$, are constant.

Bank $i$ is assumed to default when its assets fall below a certain threshold determined by its liabilities, namely
at time $\tau_i = \inf\{t: A^i_t \le \Lambda_t^i\}$, where $\Lambda^i$ is a default boundary which we now work out.
At time $t = 0$, following \cite{Lipton2015},
\begin{equation*}
	\Lambda^i_{0} = R_i \left(L_i + \sum_{j \ne i} L_{ij} \right) - \sum_{j \ne i} L_{ji},
\end{equation*}
where $R_i$ is the recovery rate of bank $i$, i.e., bank $i$ defaults if the recovery value of its liabilities, external and to other banks, exceeds the sum of its assets,
external and from other banks.

Since liabilities and recovery rates are assumed constant in time, the default boundary remains constant until some bank defaults.
If bank $k$ defaults at time $t$, the default boundary of bank $i$ jumps by (see also \cite{Lipton2015})
\begin{equation*}
	\Delta \Lambda_t^i = (1 - R_i R_k)  L_{ki}.
\end{equation*}

In the following, we assume that the banks have the same parameters, i.e., $\mu_i=\mu$, $\sigma_i=\sigma$ and $R_i = R$, for some $\mu$, $\sigma$, $R$, and $L_i = L$ and $L_{ij} = \frac{\gamma}{N}$, both constant, respectively, for some $L, \gamma>0$.
In particular, this implies that the asset value processes are exchangeable,
and we have $\Lambda_0^i=\Lambda_0$ for some $\Lambda_0$, which will allow us to take a large pool limit.

As a result, we can write $\Lambda^i_t$ as
\begin{equation*}
	\Lambda_t^i = \Lambda_0^i + \frac{\gamma}{N} \sum_{k \ne i} (1 - R^2)  \mathbbm{1}_{\{\tau_k \le t\}}.
\end{equation*}
It is more convenient to introduce the distance to default 
$Y_t^i = \log(A_t^i/ \Lambda_t^i)/\sigma$,
then
\begin{equation*}
	Y_t^i = Y_0^i + (\mu - \sigma^2/2) t + W_t^i - \frac{1}{\sigma} \log \left(1 + \frac{\gamma}{N} \sum_{k \ne i} (1 - R^2) \frac{1}{\Lambda_0}  \mathbbm{1}_{\{\tau_k \le t\}} \right).
\end{equation*}
Using the approximation $\log(1 + x) \approx x$ for small $x$ (i.e., assuming only a small proportion of the banks have yet defaulted), we get
for $t<\tau_i$
\begin{eqnarray*}
	Y_t^i &=& Y_t^0 + (\mu - \sigma^2/2) t + W_t^i - \frac{\gamma(1-R^2)}{\sigma \Lambda_0} L_t^N, \\
L_t^N &=&  \frac{1}{N} \sum_{k}   \mathbbm{1}_{\{\tau_k \le t\}}.
\end{eqnarray*}


For simplicity, we take the special case $\mu - \sigma^2/2 =0$ (in any case, this term will be small for realistic
parameter values and not have any qualitative impact on the results).

Then, using propagation of chaos as in \cite{Andreas}, one can obtain that in the limit for $N \to \infty$
all $Y_t^i$ have the same distribution as $Y$ given by\footnote{Note the slight ambiguity between the liabilities above and the loss function below, which are both denoted by $L$. It should be clear from the context and indices
applied which one is being referred to, hence for consistency with the literature we keep the notation.}

\begin{equation}
\begin{aligned}
	Y_t &= Y_0 +  W_t - \alpha L_t, \\
	L_t &= \mathbb{P}(\tau \le t),\\
	\tau &= \inf \{t \in [0, T]: Y_t \le 0 \},
\end{aligned}
\label{mckean_vlasov_eq}
\end{equation}
where $\alpha =  \frac{\gamma(1-R^2)}{\sigma \Lambda_0}$.

\cite{Andreas} and \cite{kaushansky2018simulation} considered $Y_0$ as a random variable from a given distribution with a density $f_0(x)$.
For simplicity, we assume $Y_0 = z$ a.s.\ for some $z > 0$, 
by taking $A_0^i$ and $\Lambda_0^i$ the same for all $i$, but
the results can be extended by making $A_0^i$ random and drawn from the same distribution.


The distribution of the stopped process $Y_{t \wedge \tau}$ can be written as $L_t \delta + p$, where $\delta$ is the atomic measure concentrated at 0 and
$p$ is the continuous component.
Writing the increasing process $L$ as $\alpha L_t = - \int_0^t \mu(t') \, dt'$ for some negative $\mu$, 
$p$ satisfies
\begin{equation}
\begin{aligned}
p_{t}\left( t,x;z\right) &=- \mu(t) \, p_{x}\left( t,x;z\right) +%
\frac{1}{2}p_{xx}\left( t,x;z\right) , \\ 
p\left( 0,x;z\right) &=\delta \left( x-z\right) , \\ 
p\left( t,0;z\right) &=0.%
\end{aligned}
\label{mckean_vlasov_pde}
\end{equation}%
Using the relation $L_t =  1-\int_0^\infty p(t,x;z) \, dx$, we can express $\mu$ in terms of $p$ by
\begin{eqnarray}
\label{drift}
\label{Eq1}
g(t;z) :=  \frac{{\rm d} L_t}{{\rm d} t} = -  \int_0^\infty p_t \, dx =  \mu(t) \int_0^\infty p_{x} \, dx - \frac{1}{2} \int_0^\infty p_{xx} \, dx = \frac{1}{2} p_x(t,0;z),
\end{eqnarray}
where we have used the PDE (\ref{mckean_vlasov_pde}) as well as $p(t,0;z)=0$ and  $\lim_{x\rightarrow \infty} p(t,x;z)=\lim_{x\rightarrow \infty} p_x(t,x;z)=0$. Hence, \eqref{mckean_vlasov_pde} can be written in the self-consistent form
\begin{equation}
\begin{aligned}
p_{t}\left( t,x;z\right) &= \frac{\alpha}{2} p_x(t, 0; z) \, p_{x}\left( t,x;z\right) +%
\frac{1}{2}p_{xx}\left( t,x;z\right) , \\ 
p\left( 0,x;z\right) &=\delta \left( x-z\right) , \\ 
p\left( t,0;z\right) &=0.%
\end{aligned}
\label{self-consistent}
\end{equation}%

{
From the second equation in (\ref{mckean_vlasov_eq}), $g$ is also the density of the first passage time of $Y$.
Noting a result from \cite{peskir2002integral} for the first-passage problem of Brownian motion, applied to $Y_t-L_t$,
\cite{Andreas} give the following Volterra equation for $L$:
\begin{equation*}
\Phi\left(\frac{z-\alpha L_t}{\sqrt{t}}\right) =
\int_0^t \Phi\left(\alpha \frac{L_t-L_s}{\sqrt{t-s}}\right) \, d L_s.
\end{equation*}
In contrast to this equation, we will derive a coupled system of Volterra equations which give both $p$ and $g$ (i.e., not the cumulative distribution).
These are in a more standard form without the nonlinearity $\Phi$ on the left-hand side and integration over $L$ on the right-hand side, and hence numerically more amenable.
}

In Figure \ref{L_t_fig}(a), we plot the loss function $t\rightarrow L_t$, computed by the methods described in the remainder of this paper, for different values of $\alpha$,
where $\alpha$ measures the interconnectivity of the banking system.
\begin{figure}
	\begin{center}
	\hspace{-0.6cm}
		\subfloat[]{\includegraphics[width=0.5\textwidth]{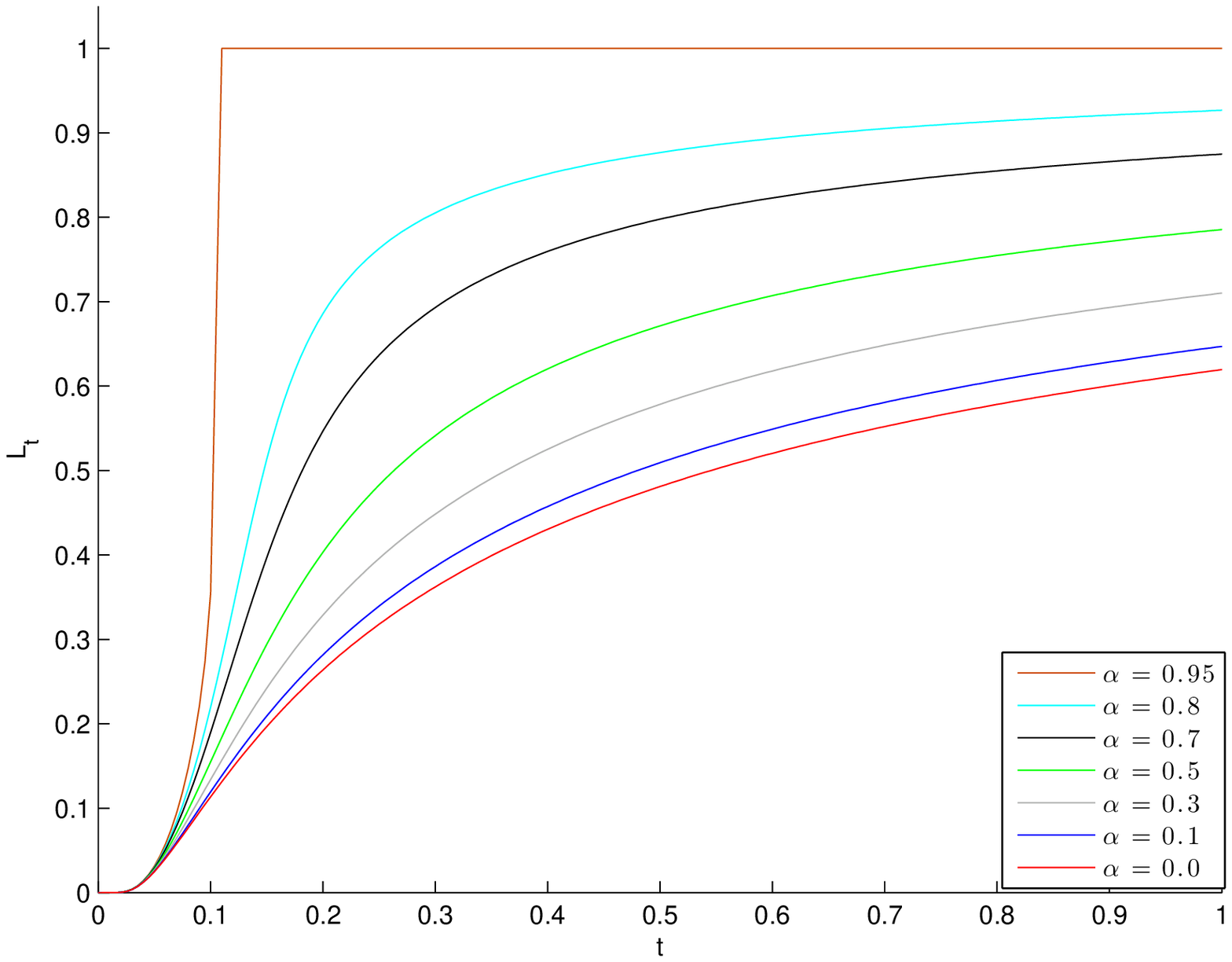}}\hspace{-0.4cm}
		\subfloat[]{\includegraphics[width=0.54\textwidth]{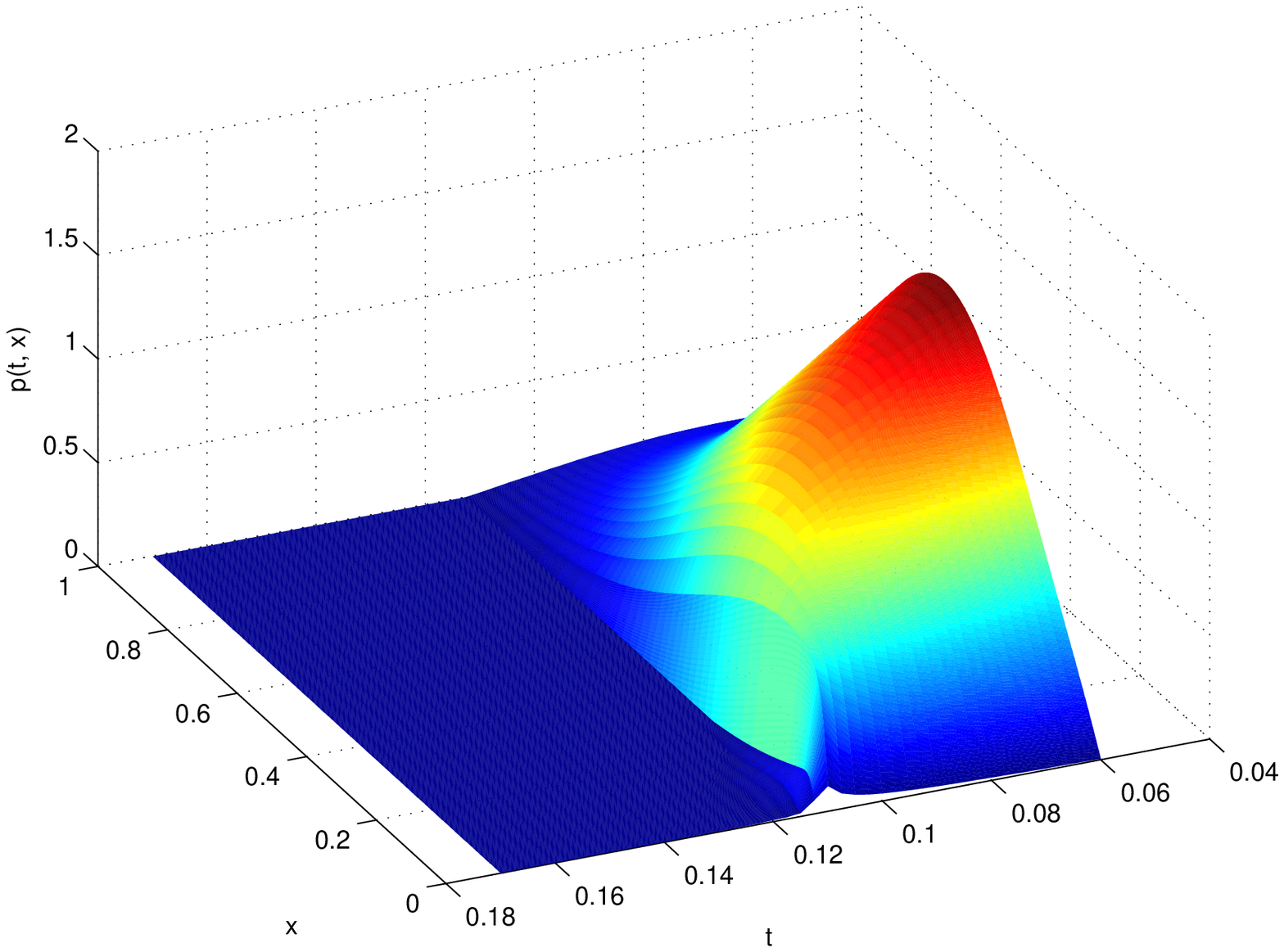}} \\
	\end{center}
	\vspace{-10pt}
	\caption{(a) $L_t$ for different values of $\alpha$ and $z=0.5$. (b) The density $p(t,x)$ for $\alpha=0.95$.}
 	\label{L_t_fig}
\end{figure}
The losses increase dramatically because of interbank liabilities, which may even lead to systemic events, here for $\alpha$ larger than around 0.9.
Hereby the rate of losses increases to infinity, as seen in Figure \ref{L_t_fig}(b) from the large gradient $p_x(t,0)$ for $t$ immediately before the blow-up,
and then triggers a jump in $L_t$.
\cite{Shkolnikov} first give a rigorous mathematical characterisation of this type of behaviour in their model, and \cite{Andreas} go further to
show the necessity of such ``blow-ups'' for large enough $\alpha$, depending on the initial distribution.

The derivation above from first principles allows us to estimate economically meaningful values of $\alpha$; see also \cite{bujok2012numerical} for the estimation of the initial values
$Y^i_0$ from credit spreads.
 According to \cite{DavidLehar}, on average, the fraction of interbank liabilities in comparison to total liabilities is $12\%$ for the EU, $8\%$ for Canada, and, as per Economic Research website of the Federal Reserve Bank of St. Louis, $4.5\%$ for the US. 
Consider, for example, the EU area. In our notation, $\sum_{j \ne i} L_{ij} \approx \gamma \approx \frac{0.12}{1 - 0.12} L \approx 0.14 L$, where $L_i=L$ are the external liabilities
for each bank, assumed identical. We can write $\alpha$ as
\begin{equation*}
	\alpha \approx \frac{1}{\sigma} \frac{(1-R^2) \gamma}{R L - (1- R) \gamma} \approx  \frac{1}{\sigma}\frac{(1-R^2) 0.14}{R  - (1- R) 0.14}.
\end{equation*}
The typical volatility of assets varies from $1\%$ to $8\%$, which can be confirmed, for example by calibration of the one-dimensional \cite{LiptonSepp2009} model. Even for a conservative case, when the recovery rate is close to $1$, we get a significant value of $\alpha$. To be precise, for $R = 0.9$ and $\sigma =0.08$, we get $\alpha \approx 0.3$.
On the other hand, 
for typical recovery rates of $R\approx 0.4$ and volatility at the lower end one can easily get $\alpha>5$.

\section{
The method of heat potentials}
\label{sec:HP}

In this section we compute the transition density and the
first passage density of Brownian motion with a known time-dependent drift $\mu(t)$ on the positive semi-axis. 
The transition probability density $p(t, x, z)$ satisfies the Kolmogorov forward equation (\ref{mckean_vlasov_pde}).

We first derive an expression for $p(t, x;z)$ using the method of heat potentials with an unknown weight function which can be found as a solution of a Volterra equation of the second kind.
Next, we differentiate the expression for $p(t, x;z)$ with respect to $x$ and take its limit to $0$ in order to find the first passage density
of $Y$ in (\ref{mckean_vlasov_eq}), or, equivalently, $g$ in (\ref{Eq1}). 
This limit is less well-known, so we calculate it for completeness.

Below we omit $z$ when possible.

\subsection{Semi-closed formula for the transition density}

Let $M\left( t\right) =\int_{0}^{t}\mu \left( t^{\prime }\right) dt^{\prime
} $. The change of variables
$\left( t,x\right) \rightarrow \left( t,y\right) =\left( t,x-M\left( t\right)\right)$  
yields
the following Cauchy-Dirichlet problem:%
\begin{equation}
\begin{aligned}
&p_{t}\left( t,y\right) =\frac{1}{2}p_{yy}\left( t,y\right) , \\ 
&p\left( 0,y\right) =\delta \left( y-z\right) , \\ 
&p\left( t,-M\left( t\right) \right) =0.%
\end{aligned}
\label{Eq2}
\end{equation}%
We split $p$ in two parts%
\begin{equation*}
p\left( t,y\right) =q\left( t,y\right) +H\left( t,y\right) ,  \label{Eq6}
\end{equation*}%
where $H\left( t,y\right) $ is the standard heat kernel,%
\begin{equation*}
H\left( t,y\right) =\frac{\exp \left( -\frac{\left( y-z\right) ^{2}}{2t}%
\right) }{\sqrt{2\pi t}}.%
\end{equation*}%
The corresponding problem for $q$ has the form%
\begin{equation*}
\begin{aligned}
 q_{t}\left( t,y\right) &=\frac{1}{2}q_{yy}\left( t,y\right) , \\ 
 q\left( 0,y\right) &=0, \\ 
 q\left( t,-M\left( t\right) \right) &=-\frac{\exp \left( -\frac{\left(
M\left( t\right) +z\right) ^{2}}{2t}\right) }{\sqrt{2\pi t}}.%
\end{aligned}
\end{equation*}

We use the method of heat potentials (see \cite{Lipton2001}, pp. 462--468). Thereby, we
represent $q$ in the form%
\begin{equation*}
\begin{aligned}
&q\left( t,y\right) =\int_{0}^{t}\frac{\left( y+M\left( t^{\prime }\right)
\right) \exp \left( -\frac{\left( y+M\left( t^{\prime }\right) \right) ^{2}}{%
2\left( t-t^{\prime }\right) }\right) }{\sqrt{2\pi \left( t-t^{\prime
}\right) ^{3}}}\nu \left( t^{\prime }\right) dt^{\prime },%
\end{aligned}
\end{equation*}%
where $\nu$ is a suitable weight which will be determined to match the boundary condition.
Assuming that $\nu \left(t\right) $ is known, we can revert to the original variables and get%
\begin{equation}
p\left( t,x\right) =\int_{0}^{t}\frac{\left( x-\Psi \left( t,t^{\prime
}\right) \right) \exp \left( -\frac{\left( x-\Psi \left( t,t^{\prime
}\right) \right) ^{2}}{2\left( t-t^{\prime }\right) }\right) }{\sqrt{2\pi
\left( t-t^{\prime }\right) ^{3}}}\nu \left( t^{\prime }\right) dt^{\prime }+%
\frac{\exp \left( -\frac{\left( x-M\left( t\right) -z\right) ^{2}}{2t},%
\right) }{\sqrt{2\pi t}},%
\label{Eq11}
\end{equation}%
where $\Psi(t, t') = M(t) - M(t')$.

By construction, $p$ in (\ref{Eq11}) satisfies the first two equations in (\ref{Eq2}).
We also need to 
satisfy the boundary condition in \eqref{Eq2}. 
It is easy to show (see Appendix \ref{app:limits}) that%
\begin{equation*}
\begin{aligned}
\mathbb{L}_{1}&:=\underset{x\rightarrow 0}{\lim }\int_{0}^{t}\frac{\left(
x-\Psi \left( t,t^{\prime }\right) \right) \exp \left( -\frac{\left( x-\Psi
\left( t,t^{\prime }\right) \right) ^{2}}{2\left( t-t^{\prime }\right) }%
\right) }{\sqrt{2\pi \left( t-t^{\prime }\right) ^{3}}}\nu \left( t^{\prime
}\right) dt^{\prime } \\ 
&=\nu \left( t\right) -\int_{0}^{t}\frac{\Psi \left( t,t^{\prime }\right) \Xi
\left( t,t^{\prime }\right) }{\sqrt{2\pi \left( t-t^{\prime }\right) ^{3}}}%
\nu \left( t^{\prime }\right) dt^{\prime }.%
\end{aligned}
\label{Eq11a}
\end{equation*}%
The requirement $\lim_{x \to 0} p(t, x) = 0$ thus leads to 
the following Volterra integral equation of the second kind for $\nu$, 
\begin{equation}
\nu \left( t\right) -\int_{0}^{t}\frac{\Psi \left( t,t^{\prime }\right) \Xi
\left( t,t^{\prime }\right) }{\sqrt{2\pi \left( t-t^{\prime }\right) ^{3}}}%
\nu \left( t^{\prime }\right) dt^{\prime }+\frac{\exp \left( -\frac{\left(
M\left( t\right) +z\right) ^{2}}{2t}\right) }{\sqrt{2\pi t}}=0,%
\label{Eq10}
\end{equation}%
where 
\begin{equation*}
\begin{aligned}
\Xi \left( t,t^{\prime }\right) &=\exp \left( -\frac{\Psi \left( t,t^{\prime
}\right) ^{2}}{2\left( t-t^{\prime }\right) }\right) , \quad t\neq t', \\ 
\Xi \left( t,t\right) &=1.%
\end{aligned}
\end{equation*}%

\subsection{Computation of loss rate over boundary}

In this section, we compute $g(t)$ from (\ref{Eq1}), suppressing $z$,
by first differenting \eqref{Eq11},
\begin{equation*}
\begin{aligned}
p_{x}\left( t,x\right) &=\int_{0}^{t}\left( 1-\frac{\left( x-\Psi \left(
t,t^{\prime }\right) \right) ^{2}}{\left( t-t^{\prime }\right) }\right) 
\frac{\exp \left( -\frac{\left( x-\Psi \left( t,t^{\prime }\right) \right)
^{2}}{2\left( t-t^{\prime }\right) }\right) }{\sqrt{2\pi \left( t-t^{\prime
}\right) ^{3}}}\nu \left( t^{\prime }\right) dt^{\prime } \\ 
&-\left( x-M\left( t\right) -z\right) \frac{\exp \left( -\frac{\left(
x-M\left( t\right) -z\right) ^{2}}{2t}\right) }{\sqrt{2\pi t^{3}}}.%
\end{aligned}
\label{Eq11c}
\end{equation*}%
In Appendix \ref{app:limits} we show that%
\begin{equation*}
\begin{aligned}
\mathbb{L}_{2}& :=
\underset{x\rightarrow 0}{\lim }\int_{0}^{t}\left( 1-\frac{\left( x-\Psi
\left( t,t^{\prime }\right) \right) ^{2}}{\left( t-t^{\prime }\right) }%
\right) \frac{\exp \left( -\frac{\left( x-\Psi \left( t,t^{\prime }\right)
\right) ^{2}}{2\left( t-t^{\prime }\right) }\right) }{\sqrt{2\pi \left(
t-t^{\prime }\right) ^{3}}}\nu \left( t^{\prime }\right) \, dt^{\prime } 
\\ 
&=2\left( M_{t}\left( t\right) -\frac{1}{\sqrt{2\pi t}}\right) \nu \left(
t\right) +\int_{0}^{t}\frac{\left( \left( 1-\frac{\Psi \left( t,t^{\prime
}\right) ^{2}}{\left( t-t^{\prime }\right) }\right) \Xi \left( t,t^{\prime
}\right) \nu \left( t^{\prime }\right) -\nu \left( t\right) \right) }{\sqrt{%
2\pi \left( t-t^{\prime }\right) ^{3}}} \, dt. 
\end{aligned}
\label{Eq11d}
\end{equation*}%
Accordingly,%
\begin{equation}
\begin{aligned}
g\left( t\right) &=\left( M_{t}\left( t\right) -\frac{1}{\sqrt{2\pi t}}%
\right) \nu \left( t\right) \\ 
&+\frac{1}{2}\int_{0}^{t}\frac{\left( \left( 1-\frac{\Psi \left( t,t^{\prime
}\right) ^{2}}{\left( t-t^{\prime }\right) }\right) \Xi \left( t,t^{\prime
}\right) \nu \left( t^{\prime }\right) -\nu \left( t\right) \right) }{\sqrt{%
2\pi \left( t-t^{\prime }\right) ^{3}}}dt^{\prime }+\frac{\left( M\left(
t\right) +z\right) \exp \left( -\frac{\left( M\left( t\right) +z\right) ^{2}%
}{2t}\right) }{2\sqrt{2\pi t^{3}}}.%
\end{aligned}
\label{Eq12}
\end{equation}

\subsection{Direct computation of loss rate}

Alternatively, we can represent 
$g\left( t\right)$ using (\ref{Eq1})
by%
\begin{equation*}
\begin{aligned}
g\left( t\right) &=-\frac{d}{dt}\int_{0}^{\infty }p\left( t,x\right) dx,%
\end{aligned}
\end{equation*}%
so that%
\begin{equation*}
\begin{aligned}
g\left( t\right) &=-\frac{d}{dt}\int_{0}^{t}\left( \int_{0}^{\infty }\frac{%
\left( x-\Psi \left( t,t^{\prime }\right) \right) \exp \left( -\frac{\left(
x-\Psi \left( t,t^{\prime }\right) \right) ^{2}}{2\left( t-t^{\prime
}\right) }\right) }{\sqrt{2\pi \left( t-t^{\prime }\right) ^{3}}}\, dx\right)
\nu \left( t^{\prime }\right) dt^{\prime } \\ 
&-\frac{d}{dt}\int_{0}^{\infty }\frac{\exp \left( -\frac{\left( x-M\left(
t\right) -z\right) ^{2}}{2t}\right) }{\sqrt{2\pi t}}\, dx \\ 
&=-\frac{d}{dt}\int_{0}^{t}\left( \int_{-\Psi \left( t,t^{\prime }\right)
}^{\infty }\frac{\xi \exp \left( -\frac{\xi ^{2}}{2\left( t-t^{\prime
}\right) }\right) }{\sqrt{2\pi \left( t-t^{\prime }\right) ^{3}}}\, d\xi
\right) \nu \left( t^{\prime }\right) dt^{\prime } \nonumber -\frac{d}{dt}\int_{-\frac{\left( M\left( t\right) +z\right) }{\sqrt{t}}%
}^{\infty }\frac{\exp \left( -\frac{\xi ^{2}}{2}\right) }{\sqrt{2\pi }} \, d\xi \\%
& =\frac{d}{dt}\int_{0}^{t}\frac{\left( \int_{-\Psi \left( t,t^{\prime
}\right) }^{\infty }d\left( \exp \left( -\frac{\xi ^{2}}{2\left( t-t^{\prime
}\right) }\right) \right) \right) }{\sqrt{2\pi \left( t-t^{\prime }\right) }}%
\nu \left( t^{\prime }\right) dt^{\prime }-\frac{d}{dt}\Phi\left( 
\frac{M\left( t\right) +z}{\sqrt{t}}\right) \\ 
&=-\frac{d}{dt}\int_{0}^{t}\frac{\Xi \left( t,t^{\prime }\right) }{\sqrt{2\pi
\left( t-t^{\prime }\right) }}\nu \left( t^{\prime }\right) dt^{\prime
}-\left( M_{t}\left( t\right) -\frac{\left( M\left( t\right) +z\right) }{2t}%
\right) \frac{\exp \left( -\frac{\left( M\left( t\right) +z\right) ^{2}}{2t}%
\right) }{\sqrt{2\pi t}}. 
\end{aligned}%
\end{equation*}

As an aside, we can verify easily by direct computation that the two expressions for $g$ agree. 
We apply the following lemma,
which will also be useful in Section \ref{subsec:pert}.

\begin{lemma}
	Consider a differentiable function $\Xi(t, t')$ such that $\Xi(t, t) = 1$. Then,
	\begin{equation*}
	\begin{aligned}
		\frac{d}{dt}\int_{0}^{t}\frac{\Xi \left( t,t^{\prime }\right) \nu \left(
t^{\prime }\right) }{\sqrt{2\pi \left( t-t^{\prime }\right) }}\, dt^{\prime } &=
\frac{\nu \left( t\right) }{\sqrt{2\pi t}}+\frac{1}{2}\int_{0}^{t}\frac{%
\nu \left( t\right) -\left( \Xi \left( t,t^{\prime }\right) -2\left(
t-t^{\prime }\right) \Xi _{t}\left( t,t^{\prime }\right) \right) \nu \left(
t^{\prime }\right) }{\sqrt{2\pi \left( t-t^{\prime }\right) ^{3}}}\, dt^{\prime
} \\
&= \int_{0}^{t}\frac{\left( \left( \Xi \left( t,t^{\prime }\right) -2\left(
t-t^{\prime }\right) \Xi _{t}\left( t,t^{\prime }\right) \right) \nu \left(
t^{\prime }\right) \right) _{t^{\prime }}}{\sqrt{2\pi \left( t-t^{\prime
}\right) }} \, dt^{\prime }.
	\end{aligned}
	\end{equation*}
\label{lemma1}
\end{lemma}
\begin{proof}
See Appendix \ref{app:prooflemma}.
\end{proof}

We use (\ref{Eq10}) and rewrite the second term in the form%
\begin{equation*}
M_{t}\left( t\right) \frac{\exp \left( -\frac{\left( M\left( t\right)
+z\right) ^{2}}{2t}\right) }{\sqrt{2\pi t}}=-M_{t}\left( t\right) \nu \left(
t\right) +M_{t}\left( t\right) \int_{0}^{t}\frac{\Psi \left( t,t^{\prime
}\right) \Xi \left( t,t^{\prime }\right) }{\sqrt{2\pi \left( t-t^{\prime
}\right) ^{3}}}\nu \left( t^{\prime }\right) dt^{\prime }.%
\end{equation*}%
Using the first equality in the lemma, we arrive at the following expression%
\begin{equation*}
\begin{aligned}
g\left( t\right) &=\left( M_{t}\left( t\right) -\frac{1}{\sqrt{2\pi t}}%
\right) \nu \left( t\right) \\ 
&-\frac{1}{2}\int_{0}^{t}\frac{\nu \left( t\right) -\left( \Xi \left(
t,t^{\prime }\right) -2\left( M_{t}\left( t\right) \Psi \left( t,t^{\prime
}\right) \Xi \left( t,t^{\prime }\right) +\left( t-t^{\prime }\right) \Xi
_{t}\left( t,t^{\prime }\right) \right) \right) \nu \left( t^{\prime
}\right) }{\sqrt{2\pi \left( t-t^{\prime }\right) ^{3}}} \, dt^{\prime } \\ 
& +\frac{\left( M\left( t\right) +z\right) \exp \left( -\frac{\left( M\left(
t\right) +z\right) ^{2}}{2t}\right) }{2\sqrt{2\pi t^{3}}}.%
\end{aligned}
\end{equation*}%
We notice that%
\begin{equation*}
\Xi _{t}\left( t,t^{\prime }\right) =\left( -\frac{M_{t}\left( t\right) \Psi
\left( t,t^{\prime }\right) }{\left( t-t^{\prime }\right) }+\frac{\Psi
\left( t,t^{\prime }\right) ^{2}}{2\left( t-t^{\prime }\right) ^{2}}\right)
\Xi \left( t,t^{\prime }\right),%
\end{equation*}%
so that%
\begin{equation*}
\begin{aligned}
& \Xi \left( t,t^{\prime }\right) -2\left( M_{t}\left( t\right) \Psi \left(
t,t^{\prime }\right) \Xi \left( t,t^{\prime }\right) +\left( t-t^{\prime
}\right) \Xi _{t}\left( t,t^{\prime }\right) \right) 
=\left( 1-\frac{\Psi \left( t,t^{\prime }\right) ^{2}}{\left( t-t^{\prime
}\right) }\right) \Xi \left( t,t^{\prime }\right),%
\end{aligned}
\end{equation*}%
from which (\ref{Eq12}) follows.

\section{Solution of the McKean-Vlasov equation}
\label{sec:solMcKV}

Now, for the McKean-Vlasov equation
(\ref{self-consistent}) 
we set in (\ref{Eq10}) and (\ref{Eq12})
\begin{equation*}
\begin{aligned}
M\left( t\right) &=-\alpha \int_{0}^{t}g\left( t^{\prime }\right) dt^{\prime
}, \\ 
\Psi \left( t,t^{\prime }\right) &=M\left( t\right) -M\left( t^{\prime
}\right) =-\alpha \int_{t^{\prime }}^{t}g\left( t^{\prime \prime }\right)
dt^{\prime  \prime}=-\alpha \Omega \left( t,t^{\prime }\right),%
\end{aligned}
\end{equation*}%
to obtain a system of
integral equations 
\begin{equation}
\left\{ 
\begin{aligned}
\nu \left( t\right) &+\int_{0}^{t}\frac{\alpha \Omega \left( t,t^{\prime
}\right) \exp \left( -\frac{\alpha ^{2}\Omega \left( t,t^{\prime }\right)
^{2}}{2\left( t-t^{\prime }\right) }\right) }{\sqrt{2\pi \left( t-t^{\prime
}\right) ^{3}}}\nu \left( t^{\prime }\right) dt^{\prime }+\frac{\exp \left( -%
\frac{\left( \alpha \Omega \left( t,0\right) -z\right) ^{2}}{2t}\right) }{%
\sqrt{2\pi t}}=0, \\ 
g\left( t\right) &+\left( \alpha g\left( t\right) +\frac{1}{\sqrt{2\pi t}}%
\right) \nu \left( t\right) + \frac{\left( \alpha \Omega \left( t,0\right) -z\right) \exp \left( -\frac{%
\left( \alpha \Omega \left( t,0\right) -z\right) ^{2}}{2t}\right) }{2\sqrt{%
2\pi t^{3}}} \\ 
&+\frac{1}{2}\int_{0}^{t}\frac{\left( \nu \left( t\right) -\left( 1-\frac{%
\alpha ^{2}\Omega \left( t,t^{\prime }\right) ^{2}}{\left( t-t^{\prime
}\right) }\right) \exp \left( -\frac{\alpha ^{2}\Omega \left( t,t^{\prime
}\right) ^{2}}{2\left( t-t^{\prime }\right) }\right) \nu \left( t^{\prime
}\right) \right) }{\sqrt{2\pi \left( t-t^{\prime }\right) ^{3}}} \, dt^{\prime } = 0.%
\end{aligned}%
\right.  
\label{Eq14}
\end{equation}

In Appendix \ref{app_special_cases}, we give the explicit solution for special cases, in particular 
when there is no feedback, $\alpha =0$, $M\left( t\right) =0$.
In general, only approximations to the solution can be found. We give an asymptotic and a numerical approach in the remainder of this section.

\subsection{Perturbation solution}
\label{subsec:pert}

We expand the solution of (\ref{Eq14}) formally in powers of $\alpha $:%
\begin{equation*}
\left( \nu \left( t\right) ,g\left( t\right) \right) =\left( \nu _{0}\left(
t\right) ,g_{0}\left( t\right) \right) +\alpha \left( \nu _{1}\left(
t\right) ,g_{1}\left( t\right) \right) +\alpha ^{2}\left( \nu _{2}\left(
t\right) ,g_{2}\left( t\right) \right) +...,  
\end{equation*}%
which we will truncate after the first two terms. This will give us an analytical expression 
which can be expected to be a good approximation for small values of $\alpha$. 

We get the following systems for $\left( \nu _{0}\left( t\right), g_{0}\left( t\right) \right) $ and $\left( \nu _{1}\left( t\right), g_{1}\left( t\right) \right) $:%
\begin{equation*}
\left\{ 
\begin{aligned}
&\nu _{0}\left( t\right) +\frac{\exp \left( -\frac{z^{2}}{2t}\right) }{\sqrt{%
2\pi t}}=0, \\ 
&g_{0}\left( t\right) +\frac{1}{\sqrt{2\pi t}}\nu _{0}\left( t\right) +\frac{1%
}{2}\int_{0}^{t}\frac{\left( \nu _{0}\left( t\right) -\nu _{0}\left(
t^{\prime }\right) \right) }{\sqrt{2\pi \left( t-t^{\prime }\right) ^{3}}}%
dt^{\prime }-\frac{z\exp \left( -\frac{z^{2}}{2t}\right) }{2\sqrt{2\pi t^{3}}%
}=0.%
\end{aligned}%
\right.   
\end{equation*}%
\begin{equation*}
\left\{ 
\begin{aligned}
& \nu _{1}\left( t\right) +\int_{0}^{t}\frac{\Omega _{0}\left( t,t^{\prime
}\right) }{\sqrt{2\pi \left( t-t^{\prime }\right) ^{3}}}\nu _{0}\left(
t^{\prime }\right) dt^{\prime }+\frac{z\Omega _{0}\left( t,0\right) \exp
\left( -\frac{z^{2}}{2t}\right) }{\sqrt{2\pi t^{3}}}=0, \\ 
& g_{1}\left( t\right) +g_{0}\left( t\right) \nu _{0}\left( t\right) +\frac{1}{%
\sqrt{2\pi t}}\nu _{1}\left( t\right) +\frac{1}{2}\int_{0}^{t}\frac{\left(
\nu _{1}\left( t\right) -\nu _{1}\left( t^{\prime }\right) \right) }{\sqrt{%
2\pi \left( t-t^{\prime }\right) ^{3}}}\, dt^{\prime } \\ 
&+\left( 1-\frac{z^{2}}{t}\right) \frac{\Omega _{0}\left( t,0\right) \exp
\left( -\frac{z^{2}}{2t}\right) }{2\sqrt{2\pi t^{3}}}=0,%
\end{aligned}%
\right.  
\end{equation*}%
where $\Omega_0 \left( t,t^{\prime }\right) =\int_{t^{\prime }}^{t}g_0\left(
t^{\prime \prime }\right) dt^{\prime  \prime}$.
The equations for $g_0$ and $g_1$ can be written as %
\begin{equation*}
\begin{aligned}
&g_{0}\left( t\right) +\int_{0}^{t}\frac{\nu _{0t}\left( t^{\prime }\right) }{%
\sqrt{2\pi \left( t-t^{\prime }\right) ^{3}}}\, dt^{\prime }-\frac{z\exp \left(
-\frac{z^{2}}{2t}\right) }{2\sqrt{2\pi t^{3}}}=0, \\ 
&g_{1}\left( t\right) +g_{0}\left( t\right) \nu _{0}\left( t\right)
+\int_{0}^{t}\frac{\nu _{1t}\left( t^{\prime }\right) }{\sqrt{2\pi \left(
t-t^{\prime }\right) }}\, dt^{\prime } +\left( 1-\frac{z^{2}}{t}\right) \frac{\Omega _{0}\left( t,0\right) \exp
\left( -\frac{z^{2}}{2t}\right) }{2\sqrt{2\pi t^{3}}}=0.%
\end{aligned}%
\end{equation*}%
Thus, using the results in Appendix \ref{app_special_cases} for $\alpha = 0$,%
\begin{equation*}
\begin{aligned}
\nu _{0}\left( t\right) &=-\frac{\exp \left( -\frac{z^{2}}{2t}\right) }{\sqrt{%
2\pi t}}, \qquad  g_{0}\left( t\right) =\frac{z\exp \left( -\frac{z^{2}}{2t}%
\right) }{\sqrt{2\pi t^{3}}} \\ 
 \Omega _{0}\left( t,t^{\prime }\right)  &= 2\left( \Phi\left( \frac{z}{\sqrt{t^{\prime }}}\right) - \Phi\left( \frac{z}{\sqrt{%
t}}\right) \right) , \\ 
\Omega _{0t}\left( t,t^{\prime }\right) &= \frac{z \exp \left( -\frac{z^{2}}{2t}%
\right) }{\sqrt{2\pi t^{3}}} = g_0(t), \qquad
\Omega _{0t'}\left( t,t^{\prime }\right)
= -g_0(t').%
\end{aligned}
\end{equation*}%
Next,%
\begin{equation}
\begin{aligned}
\nu _{1}\left( t\right) &=-\int_{0}^{t}\frac{\Omega _{0}\left( t,t^{\prime
}\right) }{\sqrt{2\pi \left( t-t^{\prime }\right) ^{3}}}\nu _{0}\left(
t^{\prime }\right) dt^{\prime }-\frac{z\Omega _{0}\left( t,0\right) \exp
\left( -\frac{z^{2}}{2t}\right) }{\sqrt{2\pi t^{3}}} \\ 
&=2\int_{0}^{t}\frac{\left( \Phi\left( \frac{z}{\sqrt{t}}\right) -\Phi\left( \frac{z%
}{\sqrt{t^{\prime }}}\right) \right) }{\sqrt{2\pi \left( t-t^{\prime
}\right) ^{3}}}\frac{\exp \left( -\frac{z^{2}}{2t^{\prime }}\right) }{\sqrt{%
2\pi t^{\prime }}}\, dt^{\prime }+\frac{2z\left( \Phi\left( \frac{z}{\sqrt{t}}%
\right) -1\right) \exp \left( -\frac{z^{2}}{2t}\right) }{\sqrt{2\pi t^{3}}},
\\ 
g_{1}\left( t\right) &=-g_{0}\left( t\right) \nu _{0}\left( t\right)
-\int_{0}^{t}\frac{\nu _{1t}\left( t^{\prime }\right) }{\sqrt{2\pi \left(
t-t^{\prime }\right) }}\, dt^{\prime }-\left( 1-\frac{z^{2}}{t}\right) \frac{%
\Omega _{0}\left( t,0\right) \exp \left( -\frac{z^{2}}{2t}\right) }{2\sqrt{%
2\pi t^{3}}},%
\end{aligned}
\label{Eq19}
\end{equation}%
where%
\begin{equation*}
\Omega _{0}\left( t,0\right) =2\left( 1-\Phi\left( \frac{z}{\sqrt{t}}\right)
\right) .  
\end{equation*}%
We can write $\nu _{1}\left( t\right) $ in the form%
\begin{equation*}
\nu _{1}\left( t\right) =-\int_{0}^{t}\frac{\omega _{0}\left( t,t^{\prime
}\right) }{\sqrt{2\pi \left( t-t^{\prime }\right) }}g_{0}\left( t^{\prime
}\right) \nu _{0}\left( t^{\prime }\right) dt^{\prime }-\frac{z\Omega
_{0}\left( t,0\right) \exp \left( -\frac{z^{2}}{2t}\right) }{\sqrt{2\pi t^{3}%
}},%
\end{equation*}%
where%
\begin{equation*}
\begin{aligned}
\omega _{0}\left( t,t^{\prime }\right) &=\frac{\Omega _{0}\left( t,t^{\prime
}\right) }{g_{0}\left( t^{\prime }\right) \left( t-t^{\prime }\right) }=-%
\frac{2\left( \Phi\left( \frac{z}{\sqrt{t}}\right) -\Phi\left( \frac{z}{\sqrt{%
t^{\prime }}}\right) \right) \sqrt{2\pi t^{\prime 3}}\exp \left( \frac{z^{2}%
}{2t^{\prime }}\right) }{z\left( t-t^{\prime }\right) }, \quad t\neq t', \\ 
\omega _{0}\left( t,t\right) &=1,%
\end{aligned}
\end{equation*}%
and obtain an expression for $\nu _{1t}$:%
\begin{equation}
\begin{aligned}
\nu _{1t}\left( t^{\prime }\right) &=-\frac{d}{dt}\int_{0}^{t}\frac{\omega
_{0}\left( t,t^{\prime }\right) }{\sqrt{2\pi \left( t-t^{\prime }\right) }}%
g_0(t') \nu _{0}\left( t^{\prime }\right) dt^{\prime }-\frac{d}{dt}\left( \frac{%
z\Omega _{0}\left( t,0\right) \exp \left( -\frac{z^{2}}{2t}\right) }{\sqrt{%
2\pi t^{3}}}\right).  
\end{aligned}
\label{eq_nu1t}
\end{equation}
To compute the first term of \eqref{eq_nu1t}, we use the second equality in Lemma \ref{lemma1} with $\nu(t) = \nu_0(t) g_0(t)$ and $\Xi(t, t') = \omega_0(t, t')$, to get
\begin{equation}
\begin{aligned}
\nu _{1t}\left( t^{\prime }\right) &=-\int_{0}^{t}\frac{\left( \left( \omega _{0}\left( t,t^{\prime }\right)
-2\left( t-t^{\prime }\right) \omega _{0t}\left( t,t^{\prime }\right)
\right) g_{0}\left( t^{\prime }\right) \nu _{0}\left( t^{\prime }\right)
\right) _{t^{\prime }}}{\sqrt{2\pi \left( t-t^{\prime }\right) }} \, dt^{\prime }
\\ 
&-\frac{d}{dt}\left( \frac{z\Omega _{0}\left( t,0\right) \exp \left( -\frac{%
z^{2}}{2t}\right) }{\sqrt{2\pi t^{3}}}\right)  \\ 
&=-\int_{0}^{t}\frac{\left( \left( \frac{3\Omega _{0}\left( t,t^{\prime
}\right) }{\left( t-t^{\prime }\right) }-2\Omega _{0t}\left( t,t^{\prime
}\right) \right) \nu _{0}\left( t^{\prime }\right) \right) _{t^{\prime }}}{%
\sqrt{2\pi \left( t-t^{\prime }\right) }} \, dt^{\prime }-\frac{d}{dt}\left( 
\frac{z\Omega _{0}\left( t,0\right) \exp \left( -\frac{z^{2}}{2t}\right) }{%
\sqrt{2\pi t^{3}}}\right)  \\ 
&=\int_{0}^{t}\frac{ \left( \frac{3 g_0(t') }{\left( t-t^{\prime }\right) } - 3 \frac{\Omega_0(t, t')}{(t - t')^2} \right) \nu_0(t') - \left(3 \frac{\Omega_0(t, t')}{(t - t')}-2g_{0}\left( t\right) \right) \nu _{0t}\left( t^{\prime }\right) }{%
\sqrt{2\pi \left( t-t^{\prime }\right) }} \, dt^{\prime } \\ 
&-\frac{z^{2}\exp \left( -\frac{z^{2}}{t}\right) }{2\pi t^{3}}+2z\left(
1-\Phi\left( \frac{z}{\sqrt{t}}\right) \right) \left( 3-\frac{z^{2}}{t}\right) 
\frac{\exp \left( -\frac{z^{2}}{2t}\right) }{2\sqrt{2\pi t^{5}}}.
\end{aligned}
\label{Eq23}
\end{equation}%
Substituting this in the second
equation (\ref{Eq19}) yields an expression for $g_{1}\left(
t\right)$, which can be evaluated by numerical integration.

Finally, we evaluate the complexity of the computation of $g_1(t)$. Consider numerical quadrature with $N$ points. First, we precompute $\nu_{1t}(t)$ using  \eqref{Eq23}; it can be done in $O(N^2)$ operations. Then, we can compute $g_1(t)$ using the second equation in \eqref{Eq19} with precomputed $\nu_{1t}(t)$ again in $O(N^2)$. Thus, the total complexity of the perturbation method is $O(N^2)$.

\subsection{Numerical solution}
\label{numerical_solution}

In this section, we present (without convergence analysis) a simple method for the numerical approximation of the solution to the coupled Volterra equations (\ref{Eq14}).
We note that Volterra equations and their numerical solution are a well-established research field. 
For a relevant discussion of the stability and convergence of some methods for equations with a weak singularity see \cite{linz1985analytical}.
\cite{noble1969instability} discusses possible instabilities of multi-step methods in the presence of weak singularities.

A number of papers propose higher order methods and collocation techniques to improve the convergence and treat instabilities. For example, \cite{brunner1985numerical} proved the convergence for a polynomial spline collocation method with quadratures; 
\cite{kolk2009high}, \cite{kolk2009numerical}, and \cite{kolk2013numerical} used a piecewise polynomial collocation method to solve a Volterra equation with weak singularity, and derived optimal global convergence estimates and a local superconvergence result. An alternative is to consider a special functional basis,
such as Chebyshev polynomials and Bernstein polynomials (see \cite{maleknejad2007numerical} and \cite{maleknejad2011new}, respectively).
In both cases, the approximation leads to a system of linear or nonlinear algebraic equations.
\cite{hairer1985fast} developed a method based on fast Fourier transform to reduce the number of kernel evaluations on an $N$-point grid from $O(N^2)$ to $O(N (\log N)^2)$.

In this paper, for simplicity, we consider trapezoidal quadrature, with a special treatment of the interval containing the singularity,
to obtain the numerical solution recursively.
We divide the interval $\left[ 0,T\right] $ into equally spaced subintervals of
length $\Delta $ and discretize (\ref{Eq14}) appropriately. To this end,
we assume that $\nu $ and $g$ are piecewise linear with $\nu(l \Delta) = \nu_l$ and $g(l \Delta) = g_l$, so that on the interval $%
\left[ \left( l-1\right) \Delta ,l\Delta \right] $ we have%
\begin{equation*}
\begin{aligned}
\nu \left( t\right) &=\frac{\nu _{l-1}\left( l\Delta -t\right) +\nu _{l}\left(
t-\left( l-1\right) \Delta \right) }{\Delta }=\nu _{l}-\frac{\left( \nu
_{l}-\nu _{l-1}\right) }{\Delta }\left( l\Delta -t\right) , \\ 
g\left( t\right) &=\frac{g_{l-1}\left( l\Delta -t\right) +g_{l}\left(
t-\left( l-1\right) \Delta \right) }{\Delta }=g_{l}-\frac{\left(
g_{l}-g_{l-1}\right) }{\Delta }\left( l\Delta -t\right) .%
\end{aligned}
\end{equation*}%
Accordingly,%
\begin{equation*}
\begin{aligned}
&\Omega _{nl}=\int_{l\Delta }^{n\Delta }g\left( t^{\prime \prime }\right)
dt^{\prime }=\frac{\Delta }{2}\left(
g_{l}+2g_{l+1}+...+2g_{n-1}+g_{n}\right) , \\ 
&\Omega _{nl}=\Omega _{\left( n-1\right) l}+\frac{\Delta }{2}\left(
g_{n}+g_{n-1}\right) .%
\end{aligned}
\end{equation*}%
Inserting in (\ref{Eq14}), the discretized system of equations has the form%
\begin{equation}
\left\{ 
\begin{aligned}
& \nu _{n}+\alpha \dsum\limits_{l=1}^{n}\mathcal{I}_{l}^{n}+\frac{\exp \left( -%
\frac{\left( \alpha \Omega _{n0}-z\right) ^{2}}{2n\Delta }\right) }{\sqrt{%
2\pi n\Delta }}=0, \\ 
& g_{n}+\left( \alpha g_{n}+\frac{1}{\sqrt{2\pi n\Delta }}\right) \nu _{n}+%
\frac{1}{2}\dsum\limits_{l=1}^{n}\mathcal{J}_{l}^{n}+\frac{\left( \alpha
\Omega _{n0}-z\right) \exp \left( -\frac{\left( \alpha \Omega _{n0}-z\right)
^{2}}{2n\Delta }\right) }{2\sqrt{2\pi n^{3}\Delta ^{3}}}=0.%
\end{aligned}%
\right.  \label{Eq24a}
\end{equation}%
For a given $n$, all the relevant integrals $\mathcal{I}_{l},\mathcal{J}%
_{l},1\leq l<n$, can be approximated by the trapezoidal rule (or via more
accurate composite formulas, if necessary). Accordingly,%
\begin{equation}
\begin{aligned}
\mathcal{I}_{l}^{n}&=\int_{\left( l-1\right) \Delta }^{l\Delta }\frac{\Omega
\left( n\Delta ,t^{\prime }\right) \exp \left( -\frac{\alpha ^{2}\Omega
\left( n\Delta ,t^{\prime }\right) ^{2}}{2\left( n\Delta -t^{\prime }\right) 
}\right) \nu \left( t^{\prime }\right) }{\sqrt{2\pi \left( n\Delta
-t^{\prime }\right) ^{3}}}\, dt^{\prime } \\ 
&\approx \frac{1}{\sqrt{8\pi \Delta }}\left( \frac{\Omega _{nl}\exp \left( -\frac{%
\alpha ^{2}\Omega _{nl}^{2}}{2\left( n-l\right) \Delta }\right) \nu _{l}}{%
\left( n-l\right) ^{3/2}}+\frac{\Omega _{n\left( l-1\right) }\exp \left( -%
\frac{\alpha ^{2}\Omega _{n\left( l-1\right) }^{2}}{2\left( n-l+1\right)
\Delta }\right) \nu _{l-1}}{\left( n-l+1\right) ^{3/2}}\right) ,%
\end{aligned}
\label{Eq24b}
\end{equation}%
and
\begin{equation}
\begin{aligned}
\mathcal{J}_{l}^{n}&=\int_{\left( l-1\right) \Delta }^{l\Delta }\frac{\left(
\nu _{n}-\left( 1-\frac{\alpha ^{2}\Omega \left( n\Delta ,t^{\prime }\right)
^{2}}{\left( n\Delta -t^{\prime }\right) }\right) \exp \left( -\frac{\alpha
^{2}\Omega \left( n\Delta ,t^{\prime }\right) ^{2}}{2\left( n\Delta
-t^{\prime }\right) }\right) \nu \left( t^{\prime }\right) \right) }{\sqrt{%
2\pi \left( n\Delta -t^{\prime }\right) ^{3}}}\, dt^{\prime } \\ 
&\approx \frac{1}{\sqrt{8\pi \Delta }}\left( \frac{\left( \nu _{n}-\left( 1-\frac{%
\alpha ^{2}\Omega _{nl}^{2}}{\left( n-l\right) \Delta }\right) \exp \left( -%
\frac{\alpha ^{2}\Omega _{nl}^{2}}{2\left( n-l\right) \Delta }\right) \nu
_{l}\right) }{\left( n-l\right) ^{3/2}}\right. \\ 
&\left. +\frac{\left( \nu _{n}-\left( 1-\frac{\alpha ^{2}\Omega _{n\left(
l-1\right) }^{2}}{\left( n-l+1\right) \Delta }\right) \exp \left( -\frac{%
\alpha ^{2}\Omega _{n\left( l-1\right) }^{2}}{2\left( n-l+1\right) \Delta }%
\right) \nu _{l-1}\right) }{\left( n-l+1\right) ^{3/2}}\right). %
\end{aligned}
\label{Eq24d}
\end{equation}

However, the last two integrals, $\mathcal{I}_{n}^{n},\mathcal{J}_{n}^{n},$
require special care, because they have weak singularities. Consider the
integral $\mathcal{I}_{n}$, which has the form%
\begin{equation*}
\mathcal{I}_{n}^{n}=\int_{\left( n-1\right) \Delta }^{n\Delta }\frac{\Omega
\left( n\Delta ,t^{\prime }\right) \exp \left( -\frac{\alpha ^{2}\Omega
\left( n\Delta ,t^{\prime }\right) ^{2}}{2\left( n\Delta -t^{\prime }\right) 
}\right) }{\sqrt{2\pi \left( n\Delta -t^{\prime }\right) ^{3}}}\nu \left(
t^{\prime }\right) dt^{\prime }.%
\end{equation*}%
In view of our piecewise linearity assumption, we have%
\begin{equation*}
\Omega \left( n\Delta ,t^{\prime }\right) =g_{n}\tau -\frac{g_{n}-g_{n-1}}{%
2\Delta }\tau ^{2},  \label{Eq26}
\end{equation*}%
where $\tau =n\Delta -t^{\prime }$. Accordingly,%
\begin{equation*}
\mathcal{I}_{n}^{n}=\int_{0}^{\Delta }\frac{\left( g_{n}-\frac{g_{n}-g_{n-1}%
}{2\Delta }\tau \right) \exp \left( -\frac{\alpha ^{2}\left( g_{n}-\frac{%
\left( g_{n}-g_{n-1}\right) }{2\Delta }\tau \right) ^{2}\tau }{2}\right)
\left( \nu _{n}-\frac{\left( \nu _{n}-\nu _{n-1}\right) }{\Delta }\tau
\right) }{\sqrt{2\pi \tau }}\, d\tau .%
\end{equation*}%
A standard change of variables $\tau =u^{2}$ yields%
\begin{multline}
\mathcal{I}_{n}^{n}=\frac{2}{\sqrt{2\pi }}\int_{0}^{\sqrt{\Delta }}\left(
g_{n}-\frac{\left( g_{n}-g_{n-1}\right) }{2\Delta }u^{2}\right) \exp \left( -%
\frac{\alpha ^{2}\left( g_{n}-\frac{\left( g_{n}-g_{n-1}\right) }{2\Delta }%
u^{2}\right) ^{2}u^{2}}{2}\right) \\ 
\times \left( \nu _{n}-\frac{\left( \nu _{n}-\nu _{n-1}\right) }{\Delta }%
u^{2}\right) du.%
\label{Eq29}
\end{multline}%
The latter integral is now non-singular and can be approximated
by the trapezoidal rule:%
\begin{equation}
\begin{aligned}
\mathcal{I}_{n}^{n}&\approx \frac{1}{\sqrt{2\pi \Delta }}\left( \Delta g_{n}\nu
_{n}+\Gamma _{n}\exp \left( -\frac{\alpha ^{2}\Gamma _{n}^{2}}{2\Delta }%
\right) \nu _{n-1}\right) 
=\sqrt{\frac{\Delta }{2\pi }}g_{n}\nu _{n}+\frac{\Gamma _{n}\exp \left( -%
\frac{\alpha ^{2}\Gamma _{n}^{2}}{2\Delta }\right) }{\sqrt{2\pi \Delta }}\nu
_{n-1},%
\end{aligned}
\label{Eq30}
\end{equation}%
where%
\begin{equation*}
\Gamma _{n}=\frac{\Delta }{2}\left( g_{n}+g_{n-1}\right) .  \label{Eq36a}
\end{equation*}%
Similarly,%
\begin{equation*}
\begin{aligned}
\mathcal{J}_{n}^{n}&=\int_{\left( n-1\right) \Delta }^{n\Delta }\frac{\left(
\nu _{n}-\left( 1-\alpha ^{2}\frac{\Omega \left( n\Delta ,t^{\prime }\right)
^{2}}{\left( n\Delta -t^{\prime }\right) }\right) \exp \left( -\frac{\alpha
^{2}\Omega \left( n\Delta ,t^{\prime }\right) ^{2}}{2\left( n\Delta
-t^{\prime }\right) }\right) \nu \left( t^{\prime }\right) \right) }{\sqrt{%
2\pi \left( n\Delta -t^{\prime }\right) ^{3}}}\, dt^{\prime } \\ 
&=\int_{\left( n-1\right) \Delta }^{n\Delta }\frac{\left( \nu _{n}-\exp
\left( -\frac{\alpha ^{2}\Omega \left( n\Delta ,t^{\prime }\right) ^{2}}{%
2\left( n\Delta -t^{\prime }\right) }\right) \nu \left( t^{\prime }\right)
\right) }{\sqrt{2\pi \left( n\Delta -t^{\prime }\right) ^{3}}}\, dt^{\prime }
\\ 
&+\alpha ^{2}\int_{\left( n-1\right) \Delta }^{n\Delta }\frac{\Omega \left(
n\Delta ,t^{\prime }\right) ^{2}\exp \left( -\frac{\alpha ^{2}\Omega \left(
n\Delta ,t^{\prime }\right) ^{2}}{2\left( n\Delta -t^{\prime }\right) }%
\right) \nu \left( t^{\prime }\right) }{\left( n\Delta -t^{\prime }\right) 
\sqrt{2\pi \left( n\Delta -t^{\prime }\right) ^{3}}}\, dt^{\prime } \\ 
&=\mathcal{J}_{n}^{n,1}+\alpha ^{2}\mathcal{J}_{n}^{n,2}.%
\end{aligned}
\end{equation*}%
We have%
\begin{align}
\mathcal{J}_{n}^{n,1}&=\int_{0}^{\Delta }\frac{\left( \nu _{n}-\exp \left( -%
\frac{\alpha ^{2}\left( g_{n}-\frac{\left( g_{n}-g_{n-1}\right) }{2\Delta }%
\tau \right) ^{2}\tau }{2}\right) \left( \nu _{n}-\frac{\left( \nu _{n}-\nu
_{n-1}\right) }{\Delta }\tau \right) \right) }{\sqrt{2\pi \tau ^{3}}}\, d\tau \nonumber \\ 
&=\int_{0}^{\Delta }\frac{\left( \left( 1-\exp \left( -\frac{\alpha
^{2}\left( g_{n}-\frac{\left( g_{n}-g_{n-1}\right) }{2\Delta }\tau \right)
^{2}\tau }{2}\right) \right) \nu _{n}\right) }{\sqrt{2\pi \tau ^{3}}}\, d\tau
\nonumber \\ 
&+\frac{\left( \nu _{n}-\nu _{n-1}\right) }{\Delta }\int_{0}^{\Delta }\frac{%
\exp \left( -\frac{\alpha ^{2}\left( g_{n}-\frac{\left( g_{n}-g_{n-1}\right) 
}{2\Delta }\tau \right) ^{2}\tau }{2}\right) }{\sqrt{2\pi \tau }}\, d\tau \nonumber \\ 
&\approx \frac{\alpha ^{2}\nu _{n}}{2}\int_{0}^{\Delta }\frac{\left( g_{n}-%
\frac{\left( g_{n}-g_{n-1}\right) }{2\Delta }\tau \right) ^{2}}{\sqrt{2\pi
\tau }}\, d\tau  \label{Eq32}\\ 
&+\frac{\left( \nu _{n}-\nu _{n-1}\right) }{\Delta }\int_{0}^{\Delta }\frac{%
\exp \left( -\frac{\alpha ^{2}\left( g_{n}-\frac{\left( g_{n}-g_{n-1}\right) 
}{2\Delta }\tau \right) ^{2}\tau }{2}\right) }{\sqrt{2\pi \tau }}\, d\tau \nonumber \\ 
&=\frac{\alpha ^{2}\nu _{n}}{\sqrt{2\pi }}\int_{0}^{\sqrt{\Delta }}\left(
g_{n}-\frac{\left( g_{n}-g_{n-1}\right) }{2\Delta }u^{2}\right) ^{2}du \nonumber \\ 
&+\frac{2\left( \nu _{n}-\nu _{n-1}\right) }{\sqrt{2\pi }\Delta }\int_{0}^{%
\sqrt{\Delta }}\exp \left( -\frac{\alpha ^{2}\left( g_{n}-\frac{\left(
g_{n}-g_{n-1}\right) }{2\Delta }u^{2}\right) ^{2}u^{2}}{2}\right) \, du \nonumber \\ 
&\approx \frac{\alpha ^{2}}{2}\frac{1}{\sqrt{2\pi \Delta ^{3}}}\left( \Delta
^{2}g_{n}^{2}+\Gamma _{n}^{2}\right) \nu _{n}+\frac{1}{\sqrt{2\pi \Delta }}%
\left( 1+\exp \left( -\frac{\alpha ^{2}\Gamma _{n}^{2}}{2\Delta }\right)
\right) \left( \nu _{n}-\nu _{n-1}\right) , \nonumber%
\end{align}%
\begin{equation}
\begin{aligned}
\mathcal{J}_{n}^{n,2}&=\int_{0}^{\Delta }\frac{\left( g_{n}-\frac{%
g_{n}-g_{n-1}}{2\Delta }\tau \right) ^{2}\exp \left( -\frac{\alpha
^{2}\left( g_{n}-\frac{\left( g_{n}-g_{n-1}\right) }{2\Delta }\tau \right)
^{2}\tau }{2}\right) \left( \nu _{n}-\frac{\left( \nu _{n}-\nu _{n-1}\right) 
}{\Delta }\tau \right) }{\sqrt{2\pi \tau }}\, d\tau \\ 
&=\frac{2}{\sqrt{2\pi }}\int_{0}^{\sqrt{\Delta }}\left( g_{n}-\frac{\left(
g_{n}-g_{n-1}\right) }{2\Delta }u^{2}\right) ^{2}\exp \left( -\frac{\alpha
^{2}\left( g_{n}-\frac{\left( g_{n}-g_{n-1}\right) }{2\Delta }u^{2}\right)
^{2}u^{2}}{2}\right) \\ 
&\times \left( \nu _{n}-\frac{\left( \nu _{n}-\nu _{n-1}\right) }{\Delta }%
u^{2}\right) du \\ 
&\approx \frac{1}{\sqrt{2\pi \Delta ^{3}}}\left( \Delta ^{2}g_{n}^{2}\nu _{n}+\Gamma
_{n}^{2}\exp \left( -\frac{\alpha ^{2}\Gamma _{n}^{2}}{2\Delta }\right) \nu
_{n-1}\right) .%
\end{aligned}
\label{Eq33}
\end{equation}%
Thus,%
\begin{equation}
\begin{aligned}
\mathcal{J}_{n}^{n}&\approx \frac{1}{\sqrt{2\pi \Delta }}\left( 1+\exp \left( -\frac{%
\alpha ^{2}\Gamma _{n}^{2}}{2\Delta }\right) \right) \left( \nu _{n}-\nu
_{n-1}\right) \\ 
&+\frac{\alpha ^{2}}{\sqrt{2\pi \Delta ^{3}}}\left( \left( \frac{3}{2}\Delta
^{2}g_{n}^{2}+\Gamma _{n}^{2}\right) \nu _{n}+\Gamma _{n}^{2}\exp \left( -%
\frac{\alpha ^{2}\Gamma _{n}^{2}}{2\Delta }\right) \nu _{n-1}\right) \\ 
&=\left( \frac{\left( 1+\exp \left( -\frac{\alpha ^{2}\Gamma _{n}^{2}}{%
2\Delta }\right) \right) }{\sqrt{2\pi \Delta }}+\frac{\alpha ^{2}\left( 
\frac{3}{2}\Delta ^{2}g_{n}^{2}+\Gamma _{n}^{2}\right) }{\sqrt{2\pi \Delta
^{3}}}\right) \nu _{n} \\ 
&-\left( \frac{\left( 1+\exp \left( -\frac{\alpha ^{2}\Gamma _{n}^{2}}{%
2\Delta }\right) \right) }{\sqrt{2\pi \Delta }}-\frac{\alpha ^{2}\Gamma
_{n}^{2}\exp \left( -\frac{\alpha ^{2}\Gamma _{n}^{2}}{2\Delta }\right) }{%
\sqrt{2\pi \Delta ^{3}}}\right) \nu _{n-1}.%
\end{aligned}
\label{Eq34}
\end{equation}

By using (\ref{Eq26}) we can represent expressions (\ref{Eq24b}), (\ref%
{Eq24d}) in a recurrent form, neglecting now quadrature errors:%
\begin{equation*}
\begin{aligned}
\mathcal{I}_{l}^{n}&=\frac{1}{\sqrt{8\pi \Delta }}\left( \frac{\Omega
_{nl}\exp \left( -\frac{\alpha ^{2}\Omega _{nl}^{2}}{2\left( n-l\right)
\Delta }\right) \nu _{l}}{\left( n-l\right) ^{3/2}}+\frac{\Omega _{n\left(
l-1\right) }\exp \left( -\frac{\alpha ^{2}\Omega _{n\left( l-1\right) }^{2}}{%
2\left( n-l+1\right) \Delta }\right) \nu _{l-1}}{\left( n-l+1\right) ^{3/2}}%
\right)  \\ 
&=\mathbb{I}_{l}^{n}\left( \left. g_{n}\right\vert \nu _{l-1},\nu
_{l},g_{1},...,g_{n-1}\right) ,%
\end{aligned}
\end{equation*}%
\begin{equation*}
\begin{aligned}
\mathcal{J}_{l}^{n}&=\frac{1}{\sqrt{8\pi \Delta }}\left( \frac{\left( \nu
_{n}-\left( 1-\frac{\alpha ^{2}\Omega _{nl}^{2}}{\left( n-l\right) \Delta }%
\right) \exp \left( -\frac{\alpha ^{2}\Omega _{nl}^{2}}{2\left( n-l\right)
\Delta }\right) \nu _{l}\right) }{\left( n-l\right) ^{3/2}}\right.  \\ 
&\left. +\frac{\left( \nu _{n}-\left( 1-\frac{\alpha ^{2}\Omega _{n\left(
l-1\right) }^{2}}{\left( n-l+1\right) \Delta }\right) \exp \left( -\frac{%
\alpha ^{2}\Omega _{n\left( l-1\right) }^{2}}{2\left( n-l+1\right) \Delta }%
\right) \nu _{l-1}\right) }{\left( n-l+1\right) ^{3/2}}\right)  \\ 
&=\left( \frac{1}{\sqrt{8\pi \Delta }\left( n-l\right) ^{3/2}}+\frac{1}{\sqrt{%
8\pi \Delta }\left( n-l+1\right) ^{3/2}}\right) \nu _{n} \\ 
&-\frac{\left( 1-\frac{\alpha ^{2}\Omega _{nl}^{2}}{\left( n-l\right) \Delta }%
\right) \exp \left( -\frac{\alpha ^{2}\Omega _{nl}^{2}}{2\left( n-l\right)
\Delta }\right) \nu _{l}}{\sqrt{8\pi \Delta }\left( n-l\right) ^{3/2}}-\frac{%
\left( 1-\frac{\alpha ^{2}\Omega _{n\left( l-1\right) }^{2}}{\left(
n-l+1\right) \Delta }\right) \exp \left( -\frac{\alpha ^{2}\Omega _{n\left(
l-1\right) }^{2}}{2\left( n-l+1\right) \Delta }\right) \nu _{l-1}}{\sqrt{%
8\pi \Delta }\left( n-l+1\right) ^{3/2}} \\ 
&=\mathbb{J}_{l}^{n}\left( \left. \nu _{n},g_{n}\right\vert \nu _{l-1},\nu
_{l},g_{1},...,g_{n-1}\right)  
=\mathbb{U}_{l}^{n}\nu _{n}+\mathbb{V}_{l}^{n}\left( \left. g_{n}\right\vert
\nu _{l-1},\nu _{l},g_{1},...,g_{n-1}\right) .%
\end{aligned}
\end{equation*}%
By the same token, $\mathcal{I}_{n}^{n}$, $\mathcal{J}_{n}^{n}$ given by 
(\ref{Eq30}), (\ref{Eq34}) can be written in the form%
\begin{equation*}
\begin{aligned}
\mathcal{I}_{n}^{n}&=\mathbb{I}_{n}^{n}\left( \left. \nu
_{n},g_{n}\right\vert \nu _{n-1},g_{n-1}\right)  \\ 
&=\mathbb{A}_{n}^{n}\left( \left. g_{n}\right\vert g_{n-1}\right) \nu _{n}+%
\mathbb{B}_{n}^{n}\left( \left. g_{n}\right\vert \nu _{n-1},g_{n-1}\right) ,%
\end{aligned}
\end{equation*}%
\begin{equation*}
\begin{aligned}
\mathcal{J}_{n}^{n}&=\mathbb{J}_{n}^{n}\left( \left. \nu
_{n},g_{n}\right\vert \nu _{n-1},g_{n-1}\right)  \\ 
&=\mathbb{U}_{n}^{n}\left( \left. g_{n}\right\vert g_{n-1}\right) \nu _{n}+%
\mathbb{V}_{n}^{n}\left( \left. g_{n}\right\vert \nu _{n-1},g_{n-1}\right). 
\end{aligned}
\end{equation*}%
In view of the above, system (\ref{Eq24a}) can be written as follows:%
\begin{equation}
\left\{ 
\begin{aligned}
&\nu _{n}+\alpha \dsum\limits_{l=1}^{n-1}\mathbb{I}_{l}^{n}\left(
g_{n}\right) +\alpha \mathbb{I}_{n}^{n}\left( \nu _{n},g_{n}\right) +\frac{%
\exp \left( -\frac{\left( \alpha \Omega _{n0}-z\right) ^{2}}{2n\Delta }%
\right) }{\sqrt{2\pi n\Delta }}=0, \\ 
& g_{n}+\left( \alpha g_{n}+\frac{1}{\sqrt{2\pi n\Delta }}\right) \nu _{n}+%
\frac{1}{2}\dsum\limits_{l=1}^{n-1}\mathbb{J}_{l}^{n}\left( \nu
_{n},g_{n}\right) +\frac{1}{2}\mathbb{J}_{n}^{n}\left( \nu _{n},g_{n}\right) 
\\ 
&+\frac{\left( \alpha \Omega _{n0}-z\right) \exp \left( -\frac{\left( \alpha
\Omega _{n0}-z\right) ^{2}}{2n\Delta }\right) }{2\sqrt{2\pi n^{3}\Delta ^{3}}%
}=0,%
\end{aligned}%
\right.   \label{Eq39}
\end{equation}%
where we suppress explicit dependencies on $g_{1},...,g_{n-1},\nu
_{1},...,\nu _{n-1}$ for brevity.
provided that $g_{1},...,g_{n-1},\nu _{1},...,\nu _{n-1}$ are given.
This
system can be solved by using the Newton-Raphson method, say. As a result,
the new pair $\left( \nu _{n},g_{n}\right) $ can be found and the recurrence
advanced by one more step as required.

If so desired, system (\ref{Eq39}) can be simplified further.  
We notice that the dependence on $\nu _{n}$ is linear, eliminate $\nu _{n}$ in
favor of $g_{n}$ from the first equation, %
\begin{equation*}
\nu _{n}=-\frac{\alpha \left( \dsum\limits_{l=1}^{n-1}\mathbb{I}%
_{l}^{n}\left( g_{n}\right) +\mathbb{B}_{n}^{n}\left( g_{n}\right) \right) +%
\frac{\exp \left( -\frac{\left( \alpha \Omega _{n0}-z\right) ^{2}}{2n\Delta }%
\right) }{\sqrt{2\pi n\Delta }}}{\left( 1+\alpha \mathbb{A}_{n}^{n}\left(
g_{n}\right) \right) },%
\end{equation*}%
and substitute this expression in the second equation, obtaining a scalar
recursive nonlinear equation of the form%
\begin{equation}
\begin{aligned}
g_{n}&+\frac{1}{2}\left( \dsum\limits_{l=1}^{n-1}\mathbb{V}_{l}^{n}\left(
g_{n}\right) +\mathbb{V}_{n}^{n}\left( g_{n}\right) \right)  
-\frac{\left( \alpha g_{n}+\frac{1}{\sqrt{2\pi n\Delta }}+\frac{1}{2}\left(
\dsum\limits_{l=1}^{n-1}\mathbb{U}_{l}^{n}+\mathbb{U}_{n}^{n}\left(
g_{n}\right) \right) \right)  }{\left( 1+\alpha \mathbb{A}%
_{n}^{n}\left( g_{n}\right) \right) }\\
& \times \left( \alpha \left( \dsum\limits_{l=1}^{n-1}%
\mathbb{I}_{l}^{n}\left( g_{n}\right) +\mathbb{B}_{n}^{n}\left( g_{n}\right)
\right) +\frac{\exp \left( -\frac{\left( \alpha \Omega _{n0}-z\right) ^{2}}{%
2n\Delta }\right) }{\sqrt{2\pi n\Delta }}\right) \\ 
&+\frac{\left( \alpha \Omega _{n0}-z\right) \exp \left( -\frac{\left( \alpha
\Omega _{n0}-z\right) ^{2}}{2n\Delta }\right) }{2\sqrt{2\pi n^{3}\Delta ^{3}}%
}=0.%
\end{aligned}
\label{Eq58}
\end{equation}%

\section{Numerical tests and results}
\label{sec:results}

In this section, we first analyse the convergence (order) of the numerical method, then test the accuracy of the first order expansion against the numerical solution,
and finally perform parameter studies (in $\alpha$) to investigate the influence of the mean-field interaction on the behaviour of the solution.

\subsection{Numerical method}

To demonstrate the performance of the discretisation scheme, we compare the solution with \eqref{Eq15a}, the analytical solution, in the case $\alpha = 0$.
\begin{figure}[H]
	\begin{center}
		\subfloat[]{\includegraphics[width=0.5\textwidth]{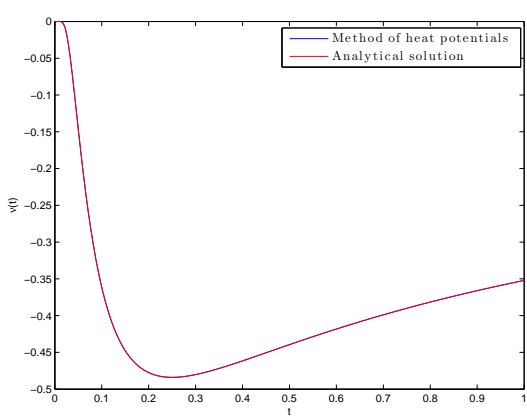}} 
		\subfloat[]{\includegraphics[width=0.5\textwidth]{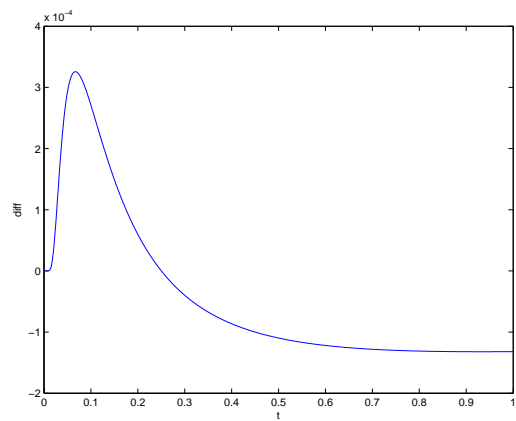}}
	\end{center}
	\vspace{-10pt}
	\caption{
	$\nu(t)$: (a) Numerical and analytical solution for $\alpha = 0$, $N=1000$ (visually indistinguishible). (b) The difference with the exact solution for $\alpha = 0$.
	}
 	\label{fig_diff}
\end{figure}

\begin{figure}[H]
	\begin{center}
		\subfloat[]{\includegraphics[width=0.5\textwidth]{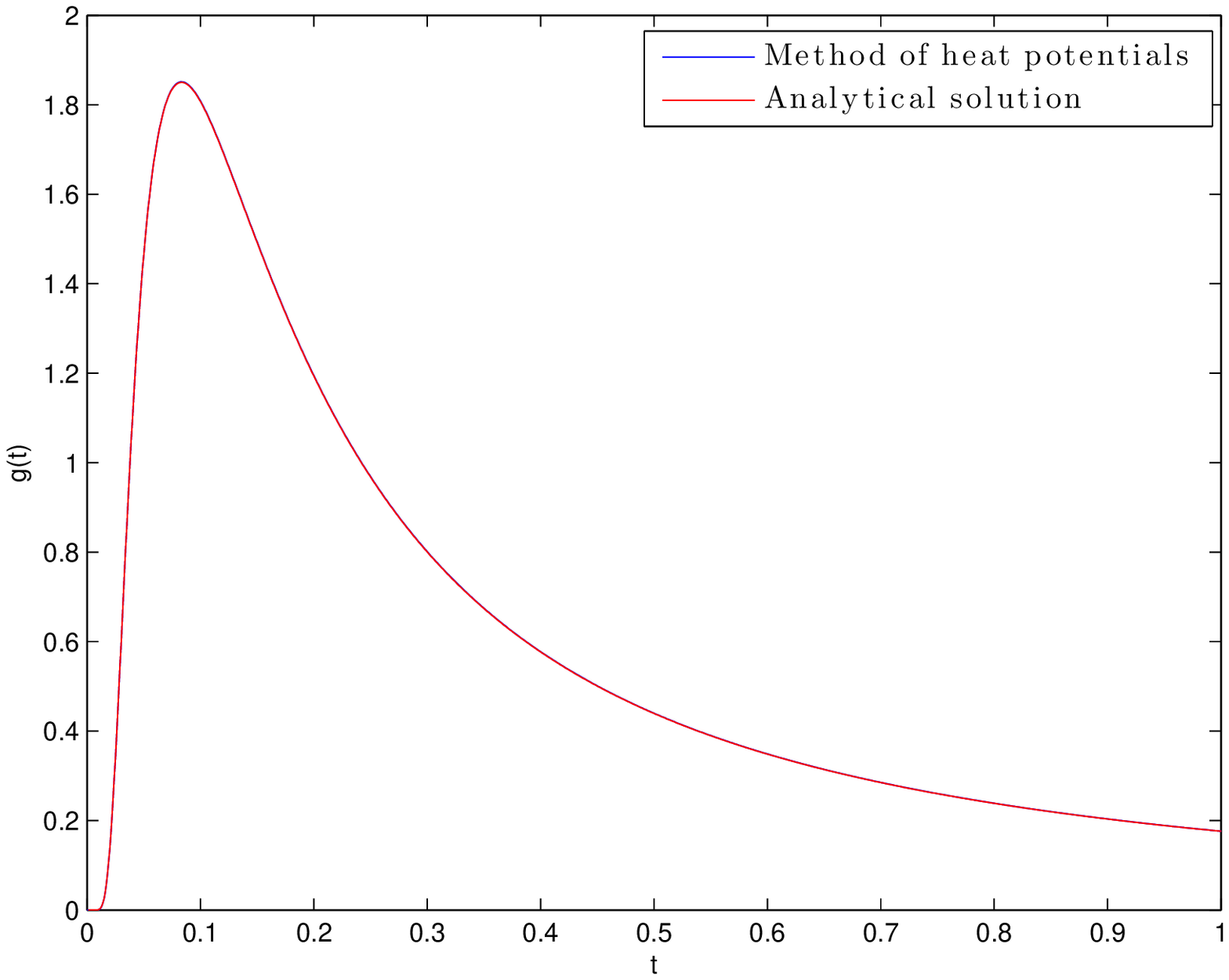}}
		\subfloat[]{\includegraphics[width=0.5\textwidth]{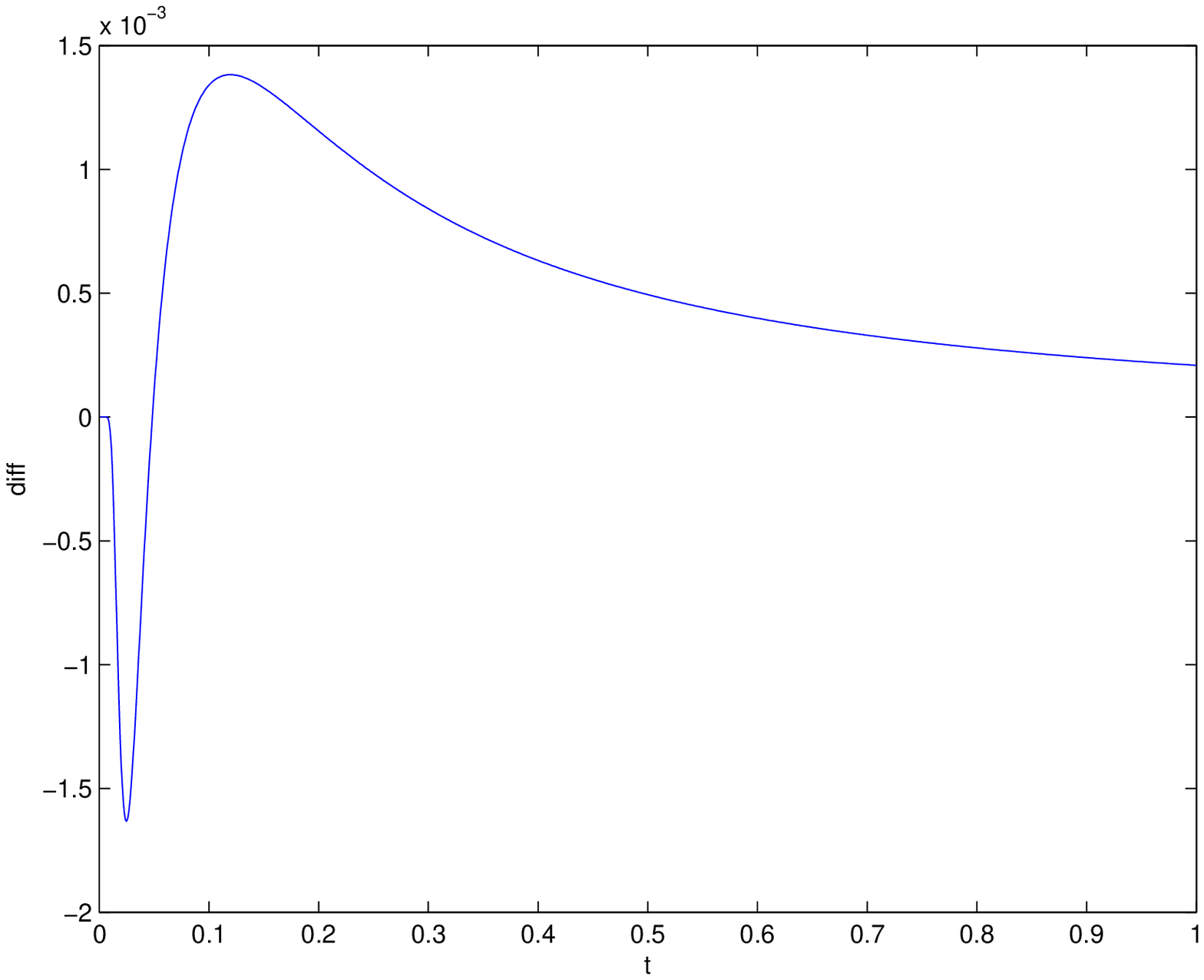}}
	\end{center}
	\vspace{-10pt}
	\caption{
	$g(t)$: (a) Numerical and analytical solution for $\alpha = 0$, $N=1000$ (visually indistinguishible). (b) The difference with the exact solution for $\alpha = 0$.}
 	\label{fig_diff1}
\end{figure}

\begin{figure}[H]
	\begin{center}
		\subfloat[]{\includegraphics[width=0.5\textwidth]{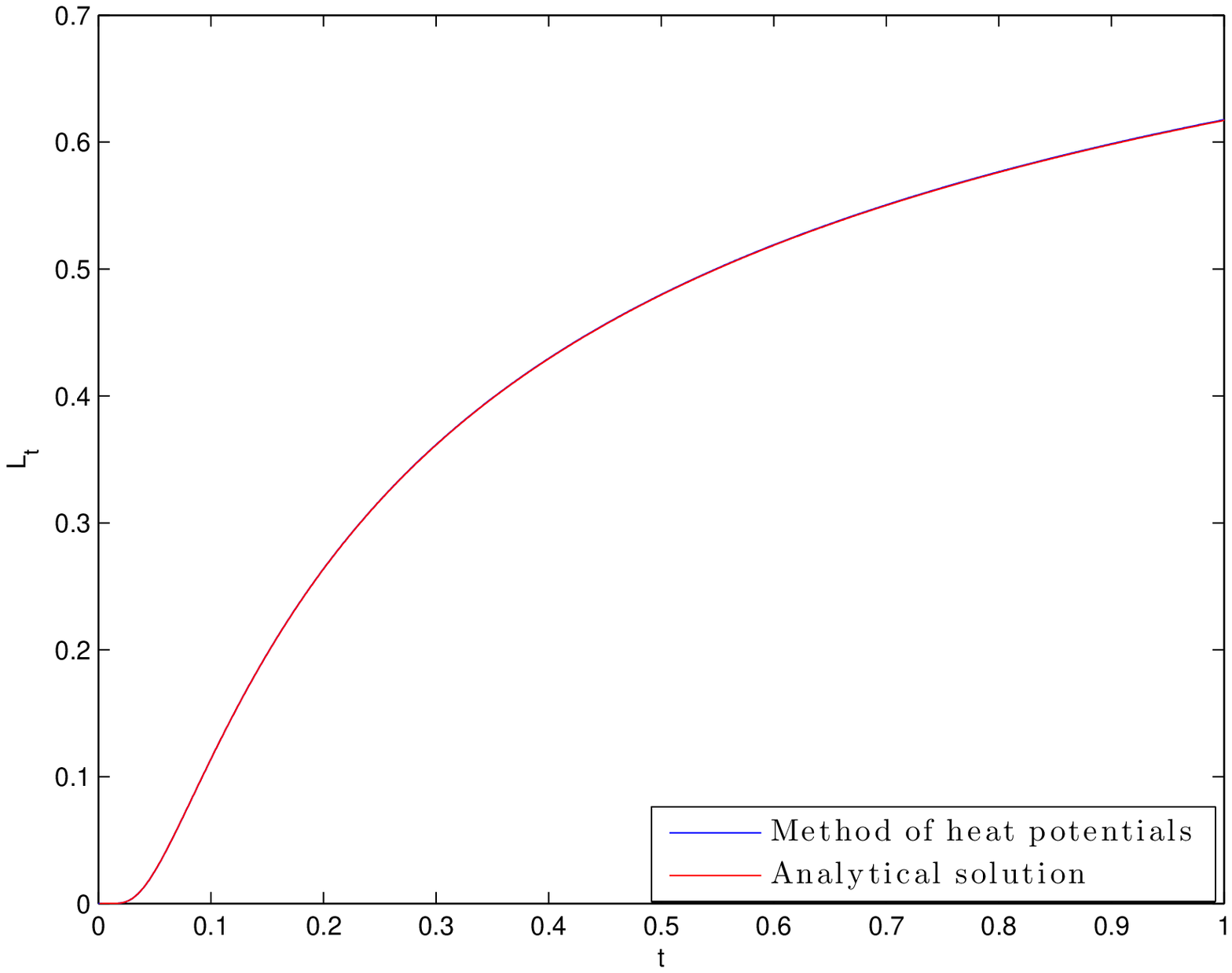}}
		\subfloat[]{\includegraphics[width=0.5\textwidth]{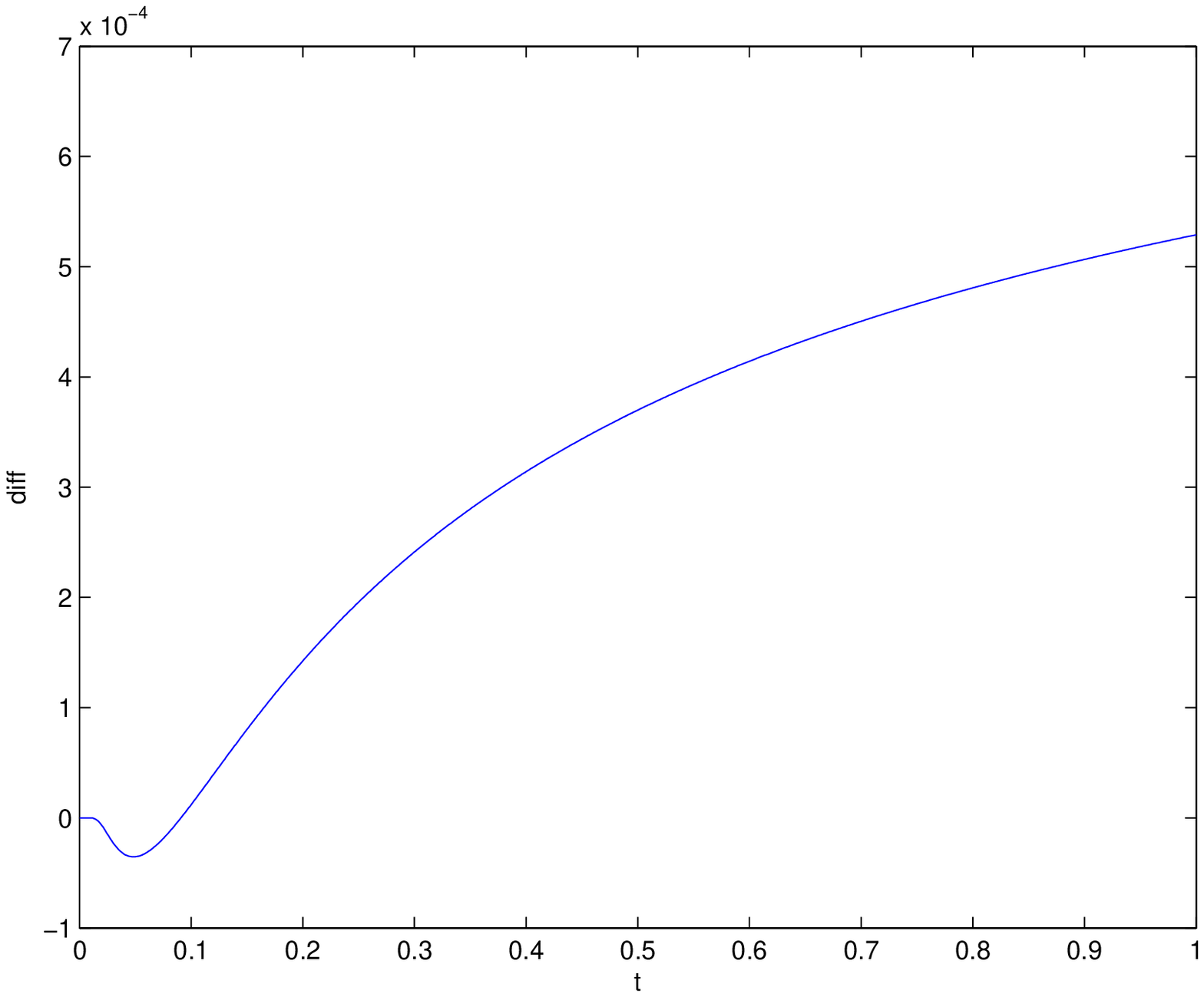}}
	\end{center}
	\vspace{-10pt}
	\caption{
	$L_t$: (a) Numerical and analytical solution for $\alpha = 0$, $N=1000$ (visually indistinguishible). (b) The difference with the exact solution for $\alpha = 0$.}
 	\label{fig_diff2}
\end{figure}

For $\alpha>0$, no closed-form solution is available and we therefore use the Euler timestepping particle method from \cite{kaushansky2018simulation}
with sufficiently many particles and timesteps as benchmark.

We illustrate the difference between our method and \cite{kaushansky2018simulation} in Figure \ref{fig_diff_MC}.
\begin{figure}[H]
	\begin{center}
		\subfloat[]{\includegraphics[width=0.5\textwidth]{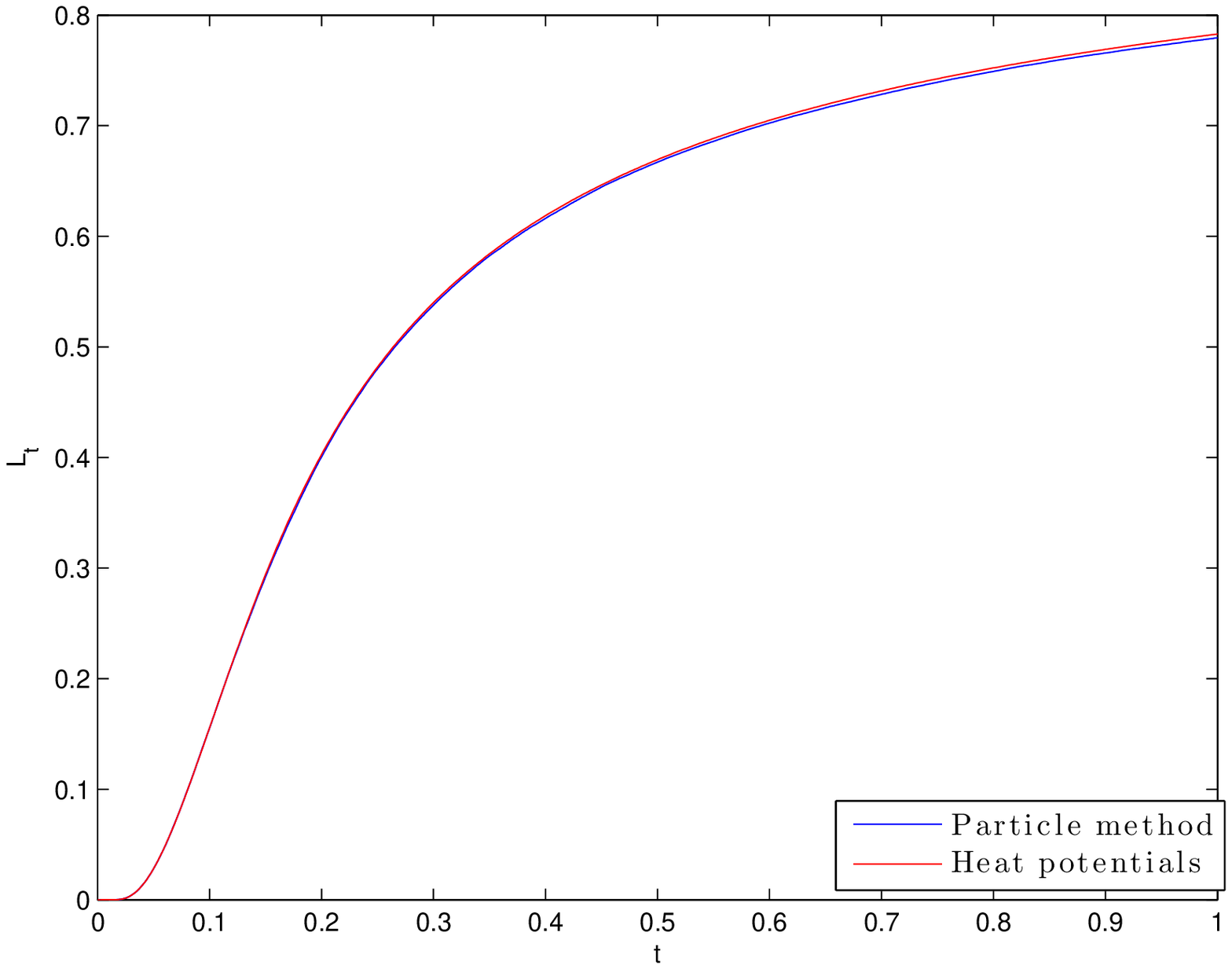}}
		\subfloat[]{\includegraphics[width=0.5\textwidth]{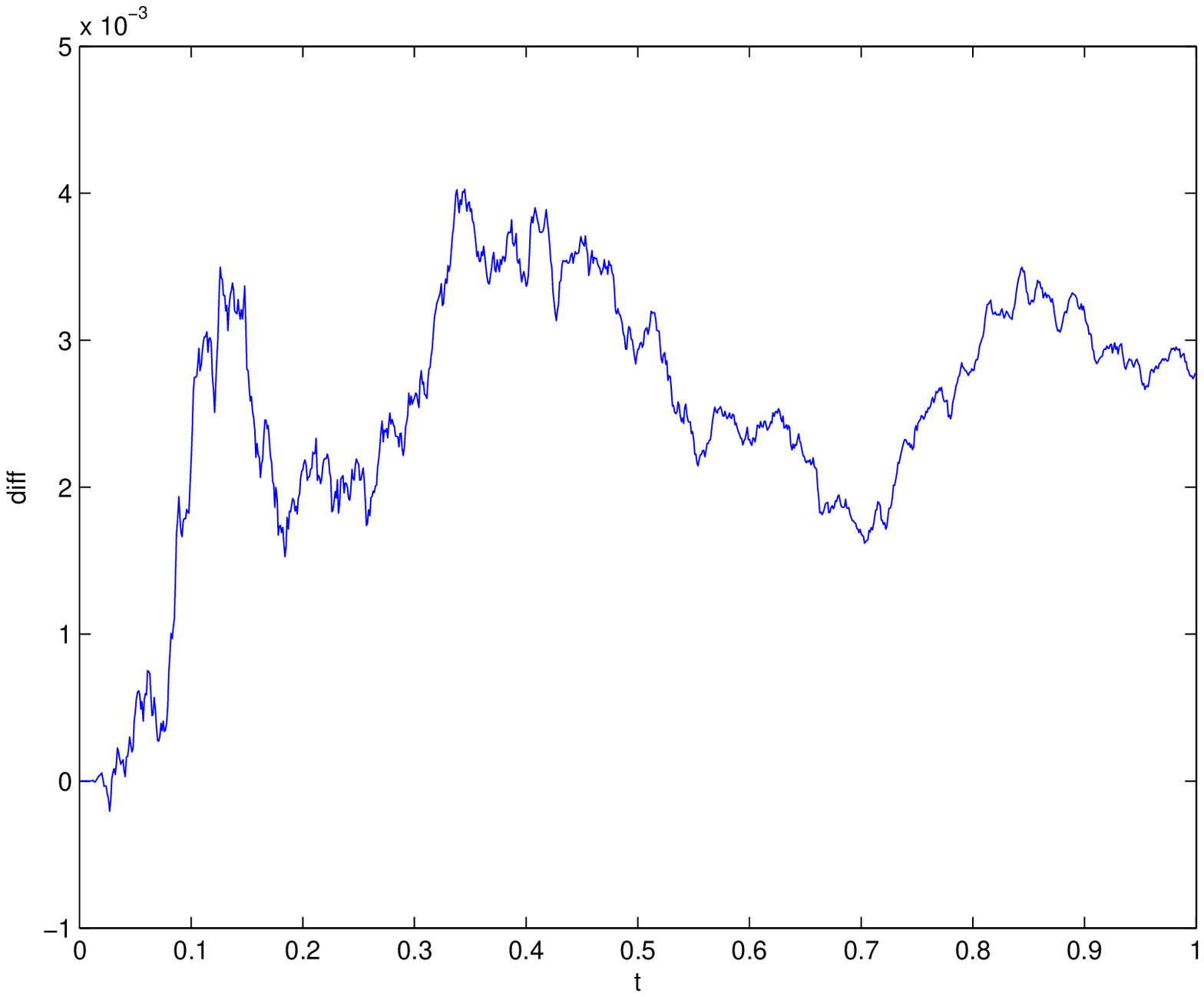}}
	\end{center}
	\vspace{-10pt}
	\caption{
	Comparison of the numerical solution by Volterra equations ($N=1000$) with that of the particle method in \cite{kaushansky2018simulation} for $\alpha = 0.5$.}
 	\label{fig_diff_MC}
\end{figure}

We now analyse the convergence order of the discretisation scheme for the Volterra equation empirically.
With $N$ time steps, the error of the approximation \eqref{Eq58} is expected to be $O(N^{-1})$, because the trapezoidal integration in \eqref{Eq29}, \eqref{Eq32}, and \eqref{Eq33} is on intervals $(0, \sqrt{\Delta})$, and the result is divided by $\sqrt{\Delta}$ after that. We empirically confirm this in Figure \ref{error_plot}. 

\begin{figure}[H]
	\begin{center}
		\subfloat[]{\includegraphics[width=0.5\textwidth]{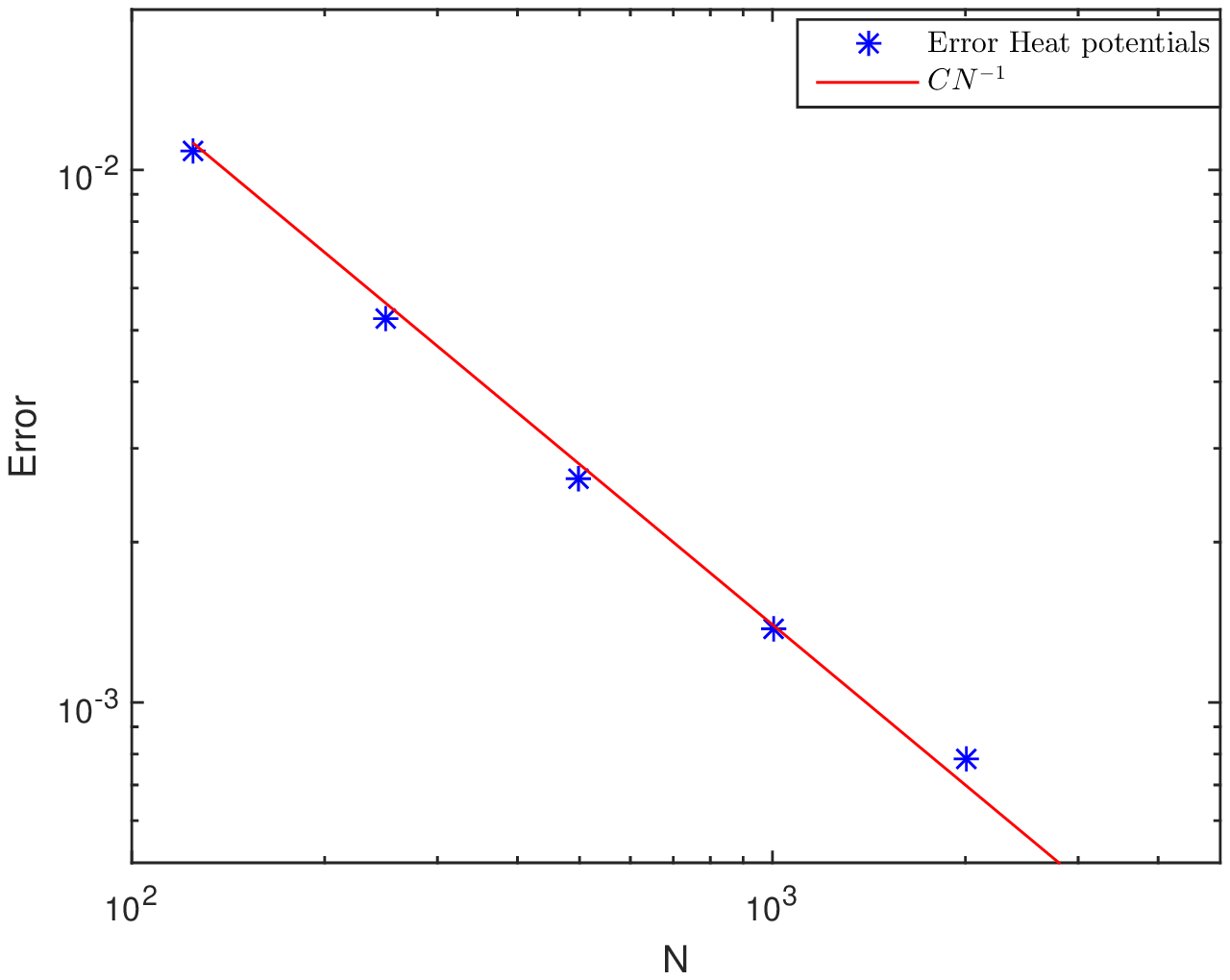}}
		\subfloat[]{\includegraphics[width=0.5\textwidth]{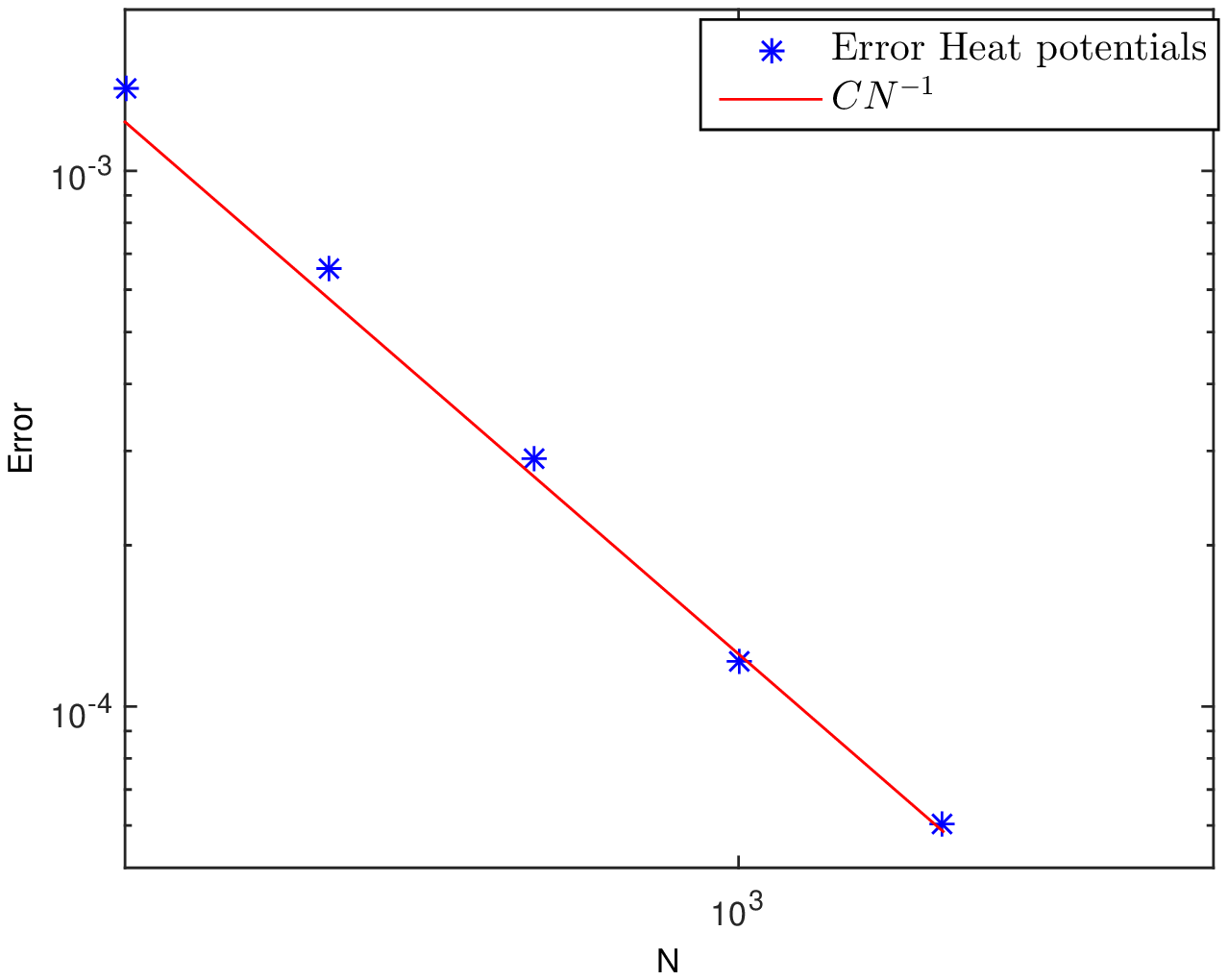}}
	\end{center}
	\vspace{-10pt}
	\caption{Error in the maximum norm for numerical method in Section \ref{numerical_solution}: (a) for $g$ compared to the exact solution for $\alpha = 0$;
	(b) for $L$ compared to the simulation solution for $\alpha = 0.5$.}
	 \label{error_plot}
\end{figure}

The complexity of our method is $O(N^2)$. Hence, in order to achieve precision $\varepsilon$, we need $O(\varepsilon^{-2})$ operations. In comparison, the particle method with Brownian bridge in \cite{kaushansky2018simulation} requires $O(\varepsilon^{-3})$ operations. The latter could be improved to $O(\varepsilon^{-2})$ by multilevel simulation,
but equally a higher order method for the Volterra equation would bring the complexity down.
Another advantage of the method above is that we automatically get directly the derivative $g$ of the loss function, which is harder to obtain in \cite{kaushansky2018simulation} because of Monte Carlo noise.

\subsection{Comparison of perturbation and numerical methods}

Here, we compare the numerical and perturbation solutions described above.
We fix $T = 1$, $z = 0.5$, and choose $N = 1000$, the number of grid points, sufficiently large so that the numerical error is negligible.
In Figure \ref{fig_num_vs_pert} we plot $g$, the hitting time density, computed with numerical and perturbation methods as well as $g_0(t)$, the solution for $\alpha = 0$, to measure the impact of the nonlinear term, for different values of $\alpha$. For $\alpha = 0.1$, the two solutions are visually indistinguishable; for $\alpha = 0.3$ there is small but visible difference between the solutions, which increases further for $\alpha = 0.5$ and arises from the higher order terms.\footnote{In our implementation of the perturbation solution, we also perform a scaling to ensure the correct cumulative density at $T$ (which in practice is unknown) to improve the results slightly.}
For $\alpha = 1$, where the numerical solution shows a jump in the loss function at around $t=0.1$, the first order expansion approximation breaks down.

\begin{figure}
	\begin{center}
		\subfloat[]{\includegraphics[width=0.5\textwidth]{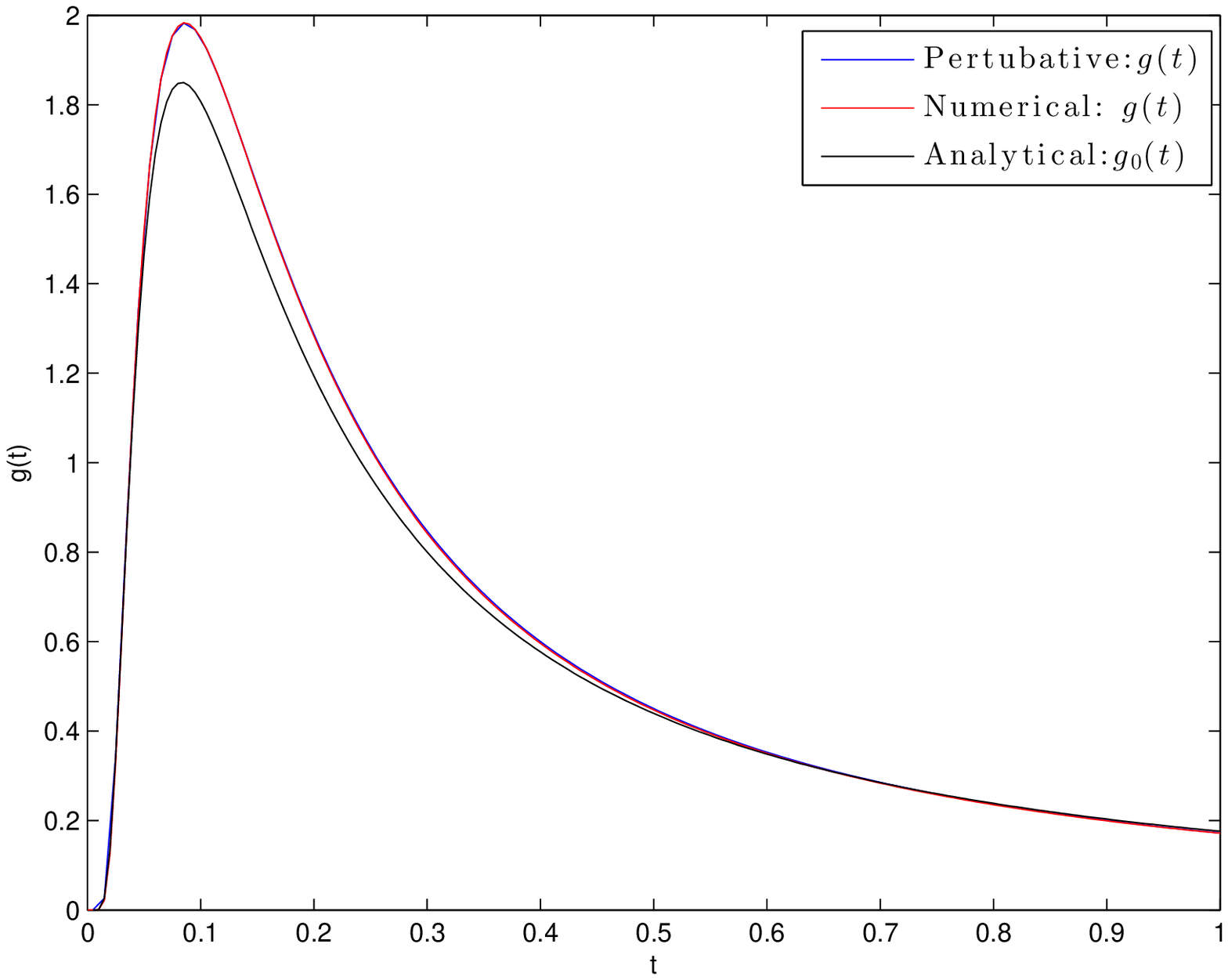}}
		\subfloat[]{\includegraphics[width=0.5\textwidth]{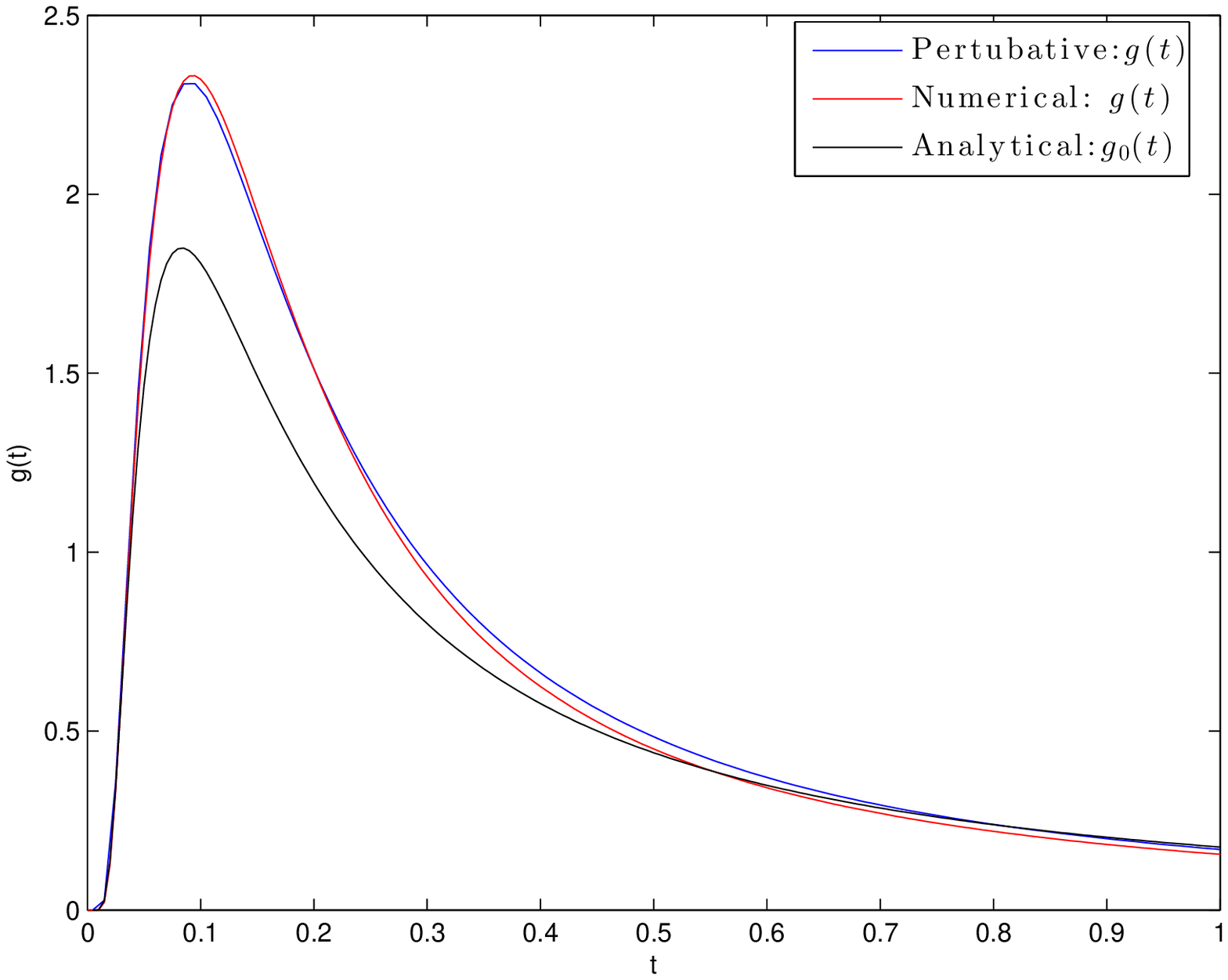}} \\
		\subfloat[]{\includegraphics[width=0.5\textwidth]{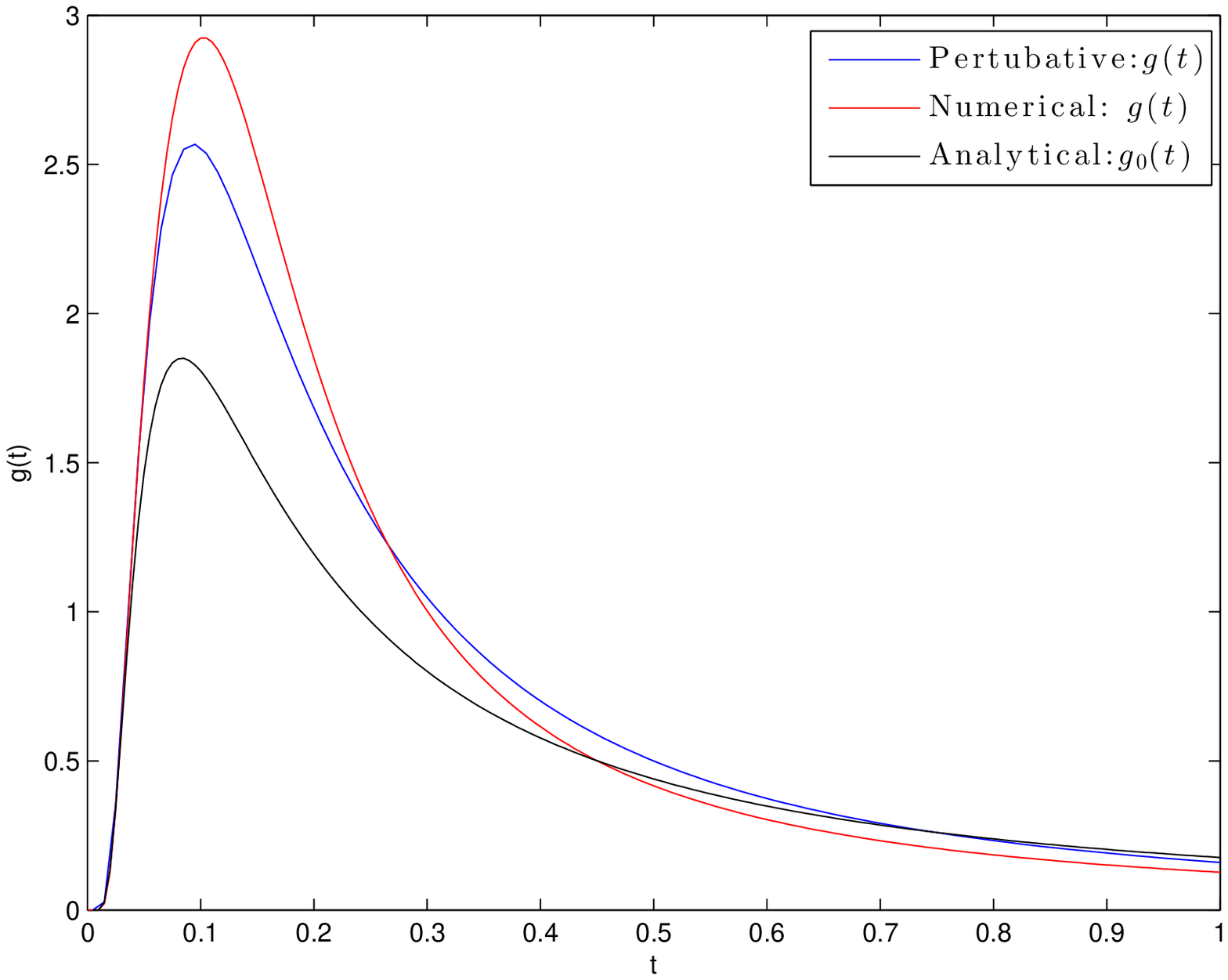}}
		\subfloat[]{\includegraphics[width=0.5\textwidth]{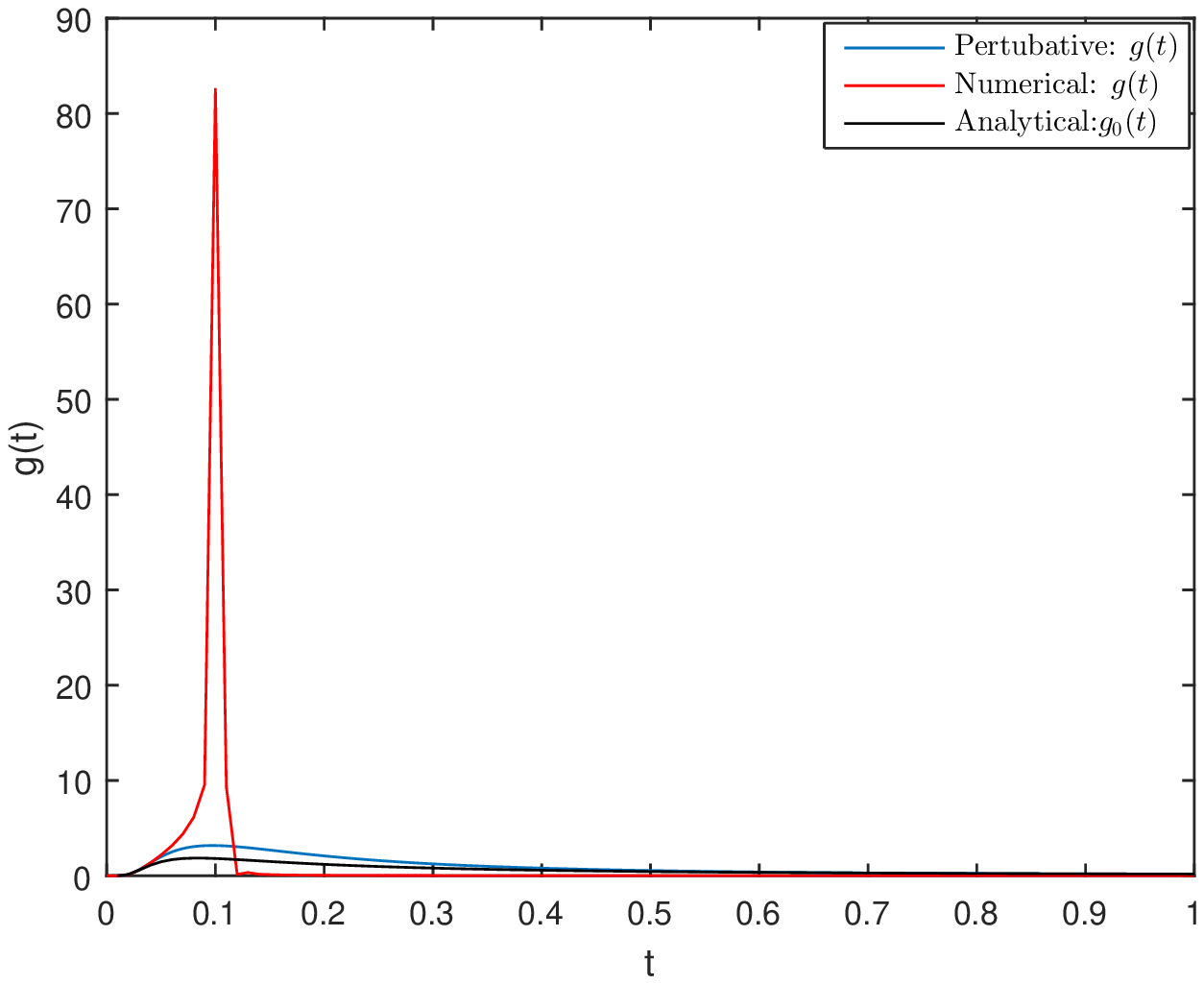}}
	\end{center}
	\vspace{-10pt}
	\caption{
	Comparison of the numerical and perturbative methods for different values of $\alpha$.: (a)  $\alpha = 0.1$;
	(b) $\alpha = 0.3 $;
	(c) $\alpha = 0.5$; (d) $\alpha = 1$. For visibility, we take $N = 100$ in the numerical method in the last plot.}
 	\label{fig_num_vs_pert}
\end{figure}

\subsection{Parameter studies}

We now assess the impact of the mean-field interaction by varying the parameter $\alpha$.

We fix $T = 1, z = 0.5$, and choose $N = 1000$, the number of the grid points.
Figure \ref{fig_alpha} demonstrates the behavior of $\nu(t)$ and $g(t)$ for different values of $\alpha$, starting with $\alpha = 0$;
for $L_t$, including a case with discontinuity, see already Figure \ref{L_t_fig}.
\begin{figure}
	\begin{center}
		\subfloat[]{\includegraphics[width=0.5\textwidth]{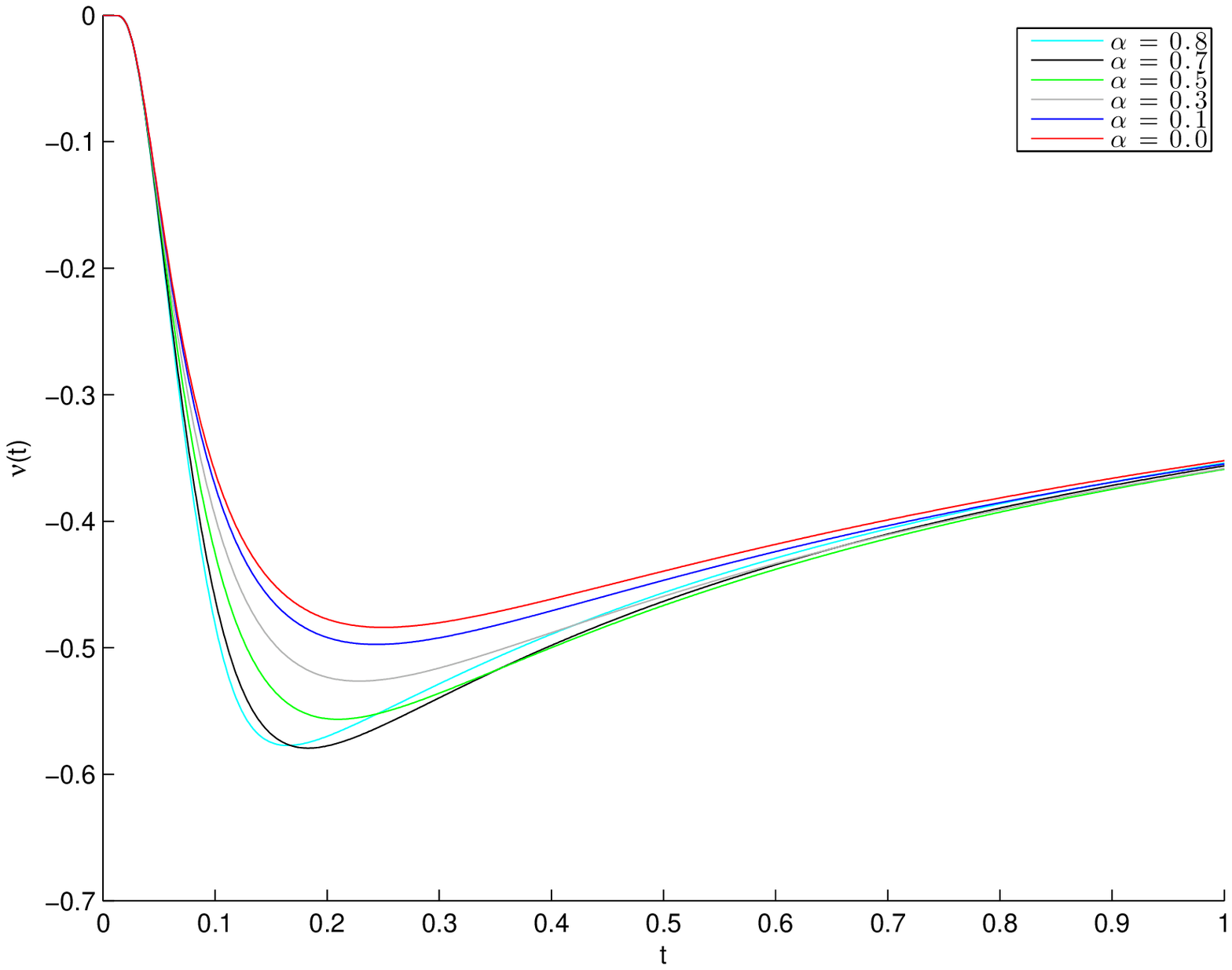}}
		\subfloat[]{\includegraphics[width=0.5\textwidth]{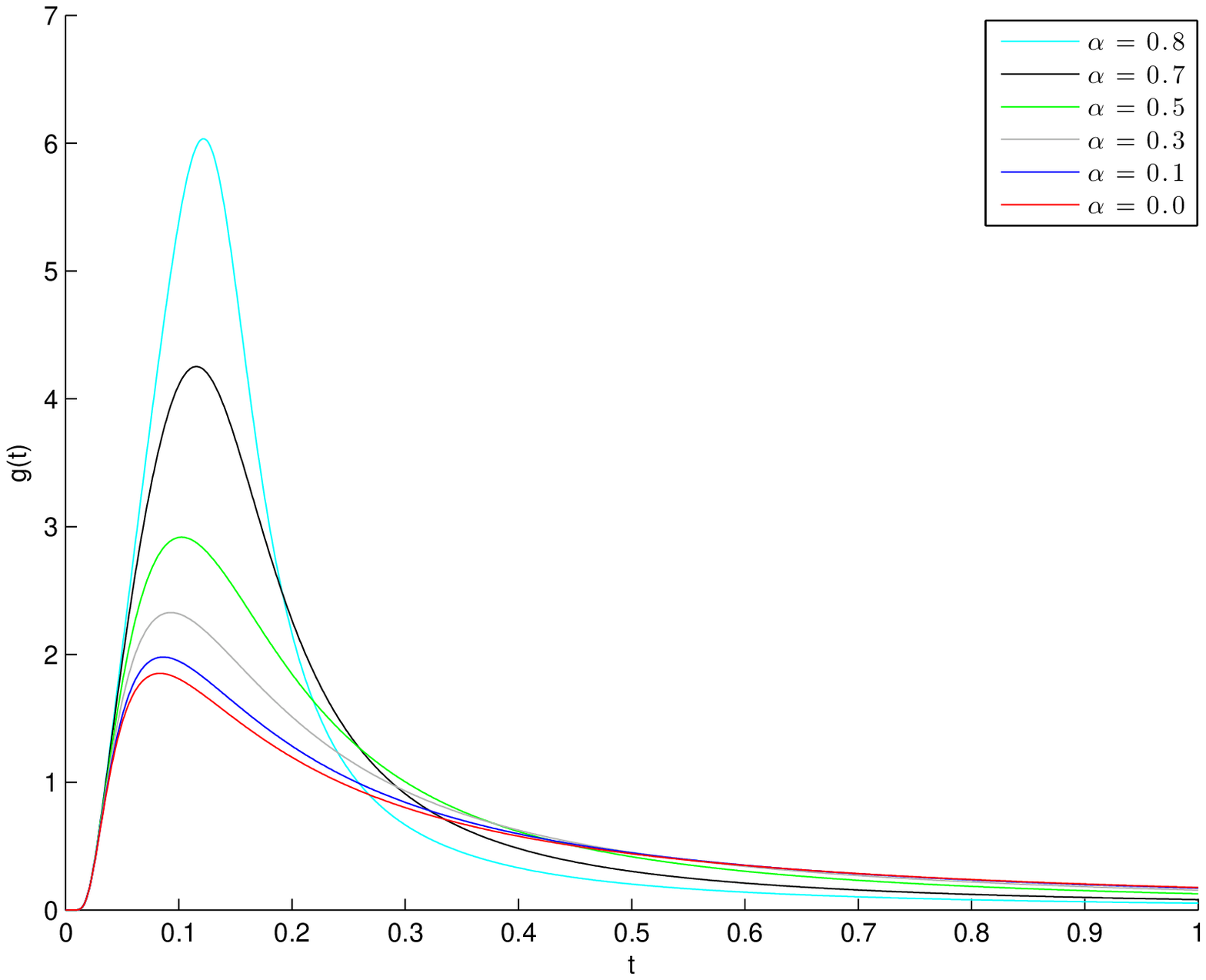}} \\
	\end{center}
	\vspace{-10pt}
	\caption{
	(a) $\nu(t)$ for different values of $\alpha$.
	(b) $g(t)$ for different values of $\alpha$.}
 	\label{fig_alpha}
\end{figure}

To illustrate the impact of the interaction further, we consider the expectation and variance of the default time depending on $\alpha$. Since the expectation is infinite, we restrict it to the interval $(0, T)$, i.e.\ consider $\mathbb{E}[\tau | \tau < T]$ and $\mathbb{V}[\tau | \tau < T]$. These expectations  must be finite for any fixed $T$, and go to infinity when $T \to \infty$. The conditional density is then $p_{\tau | \tau < T}(t) = \frac{p_{\tau}(t)}{\int_0^T p_{\tau}(s) \, ds}$ for $t \in [0, T]$. Using this, one can easily evaluate the moments numerically. We present the results in Figure \ref{fig_moments}. As expected, we observe that the expected default time and its variance become smaller with increasing of $\alpha$, and grow with increasing $T$.

\begin{figure}
	\begin{center}
		\subfloat[]{\includegraphics[width=0.5\textwidth]{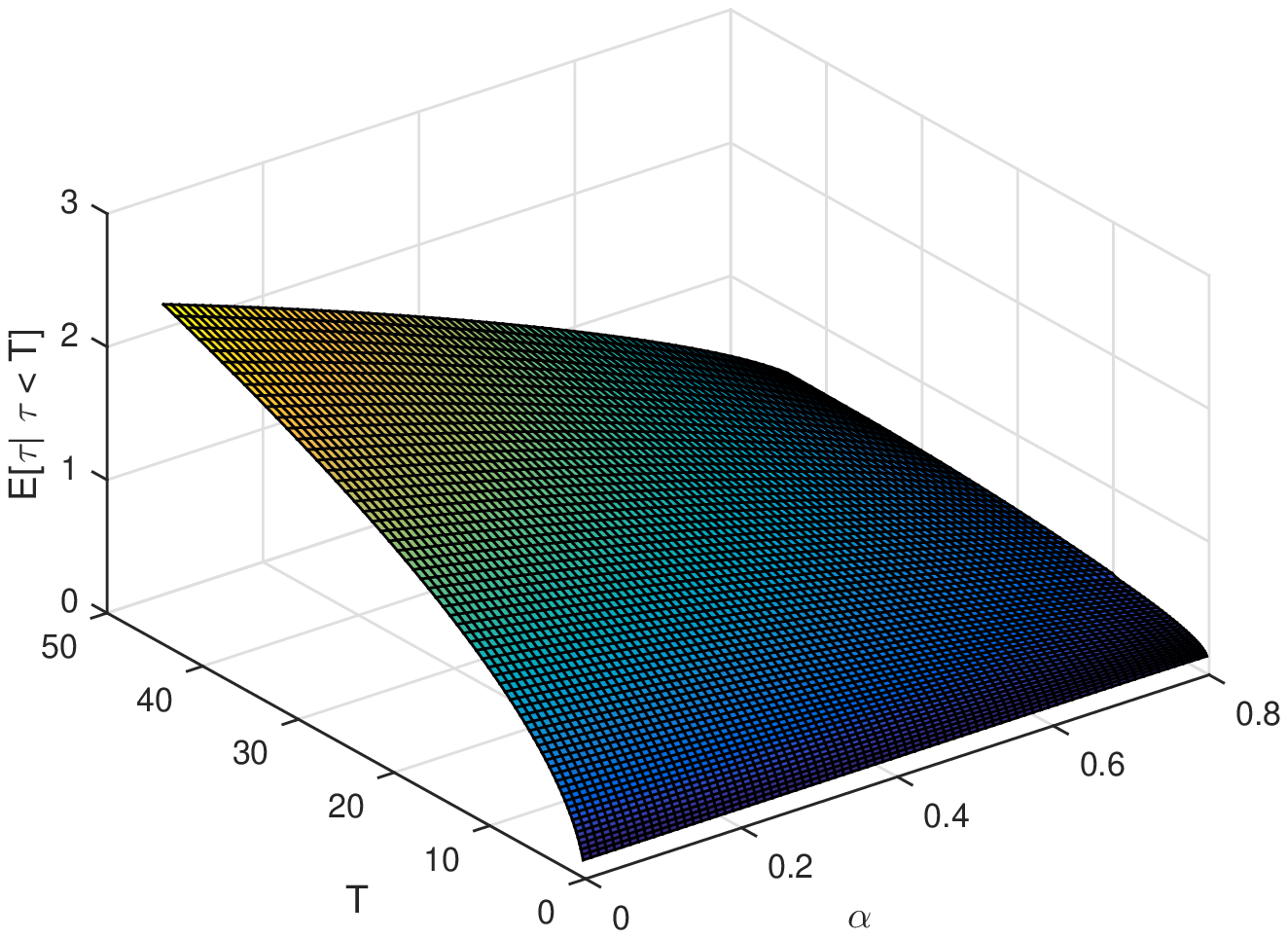}}
		\subfloat[]{\includegraphics[width=0.5\textwidth]{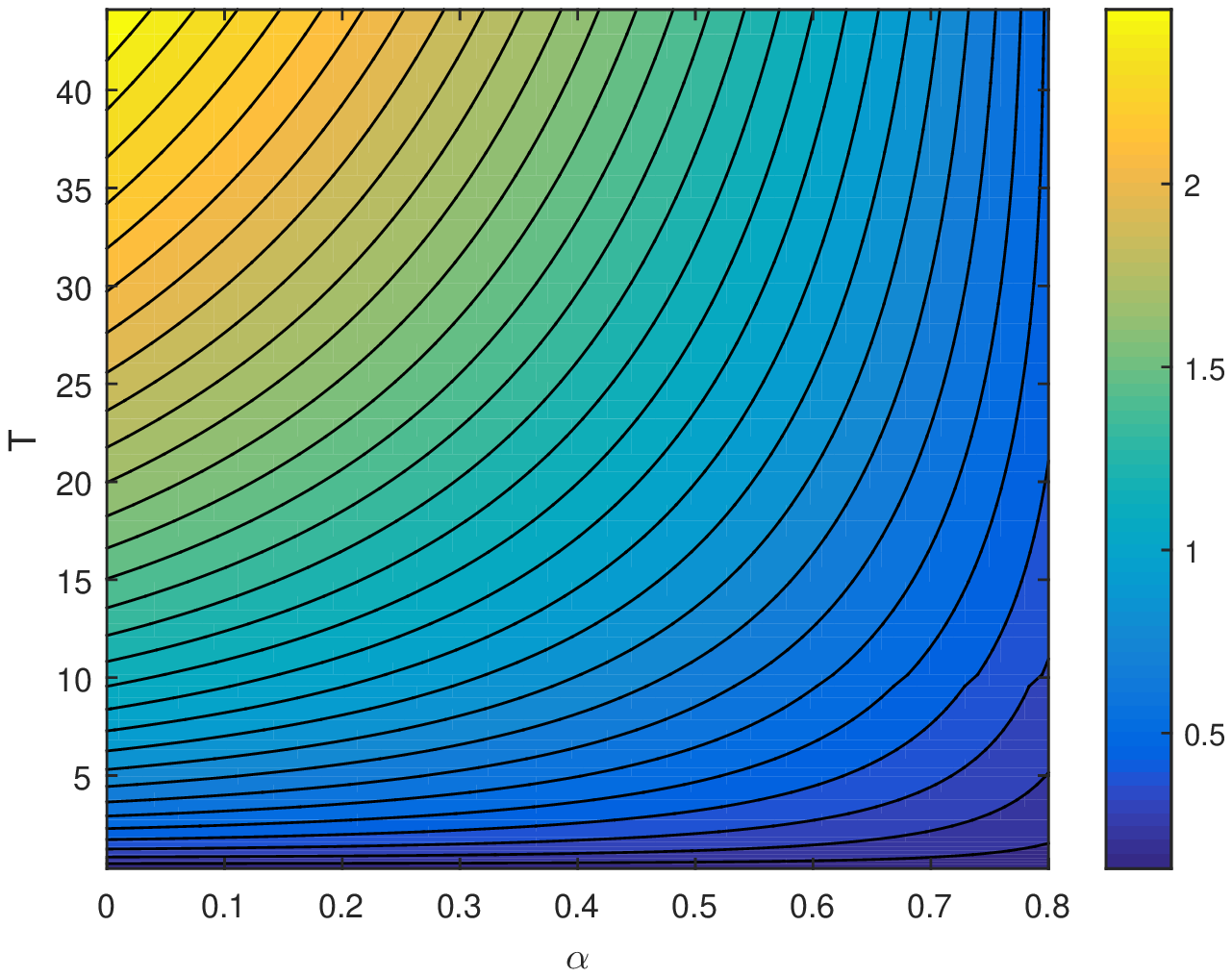}} \\
		\subfloat[]{\includegraphics[width=0.5\textwidth]{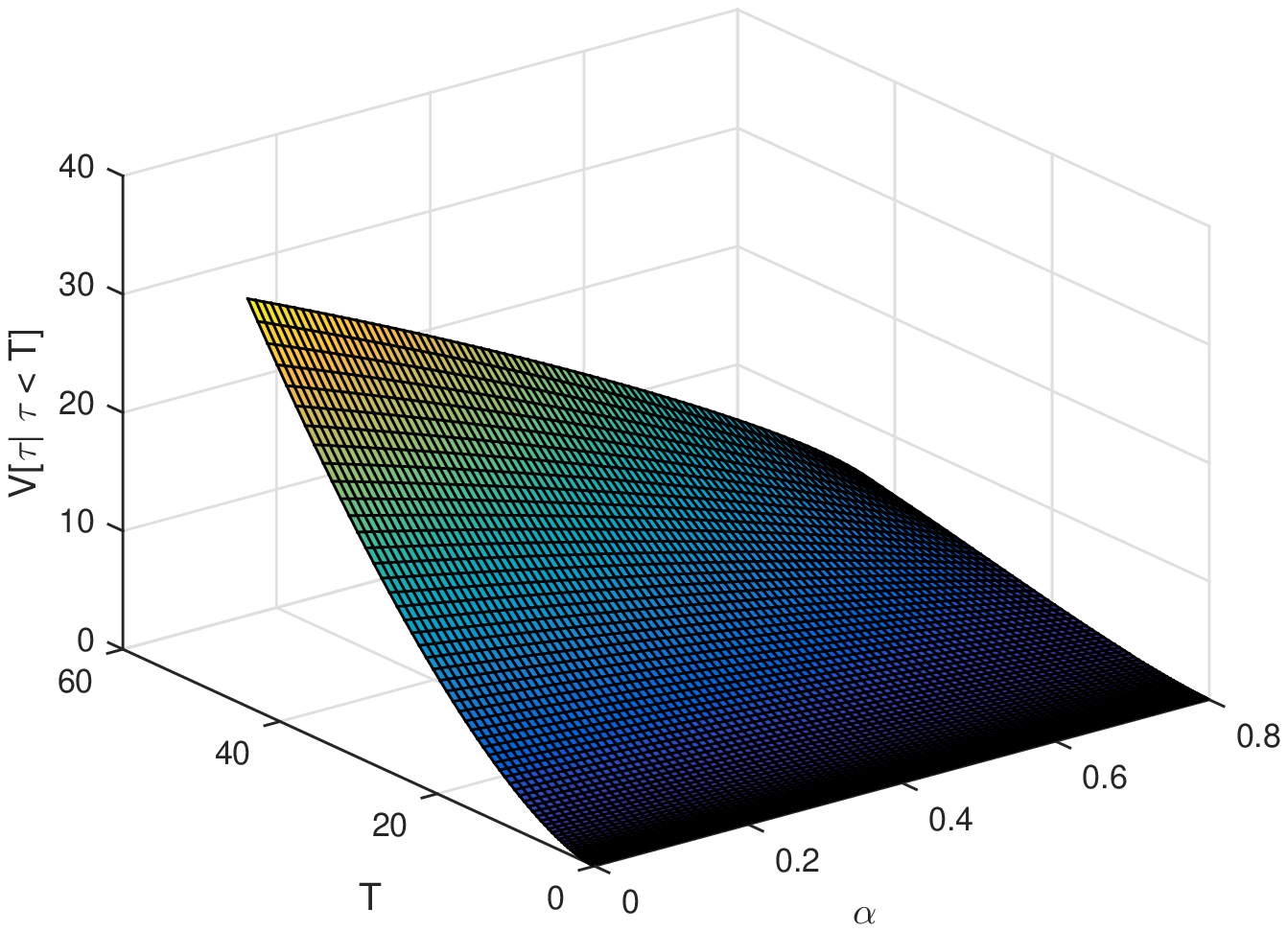}}
		\subfloat[]{\includegraphics[width=0.5\textwidth]{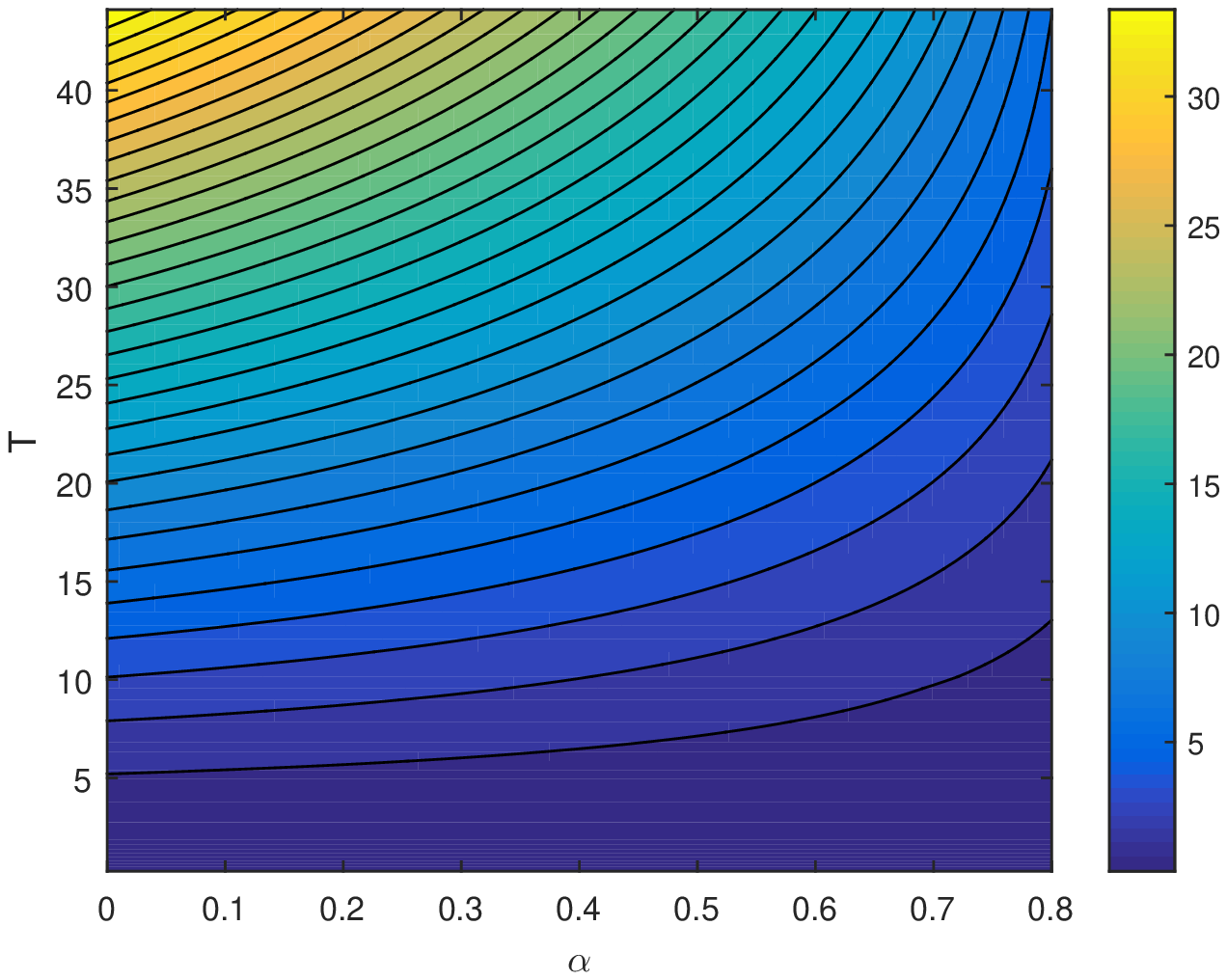}} \\
	\end{center}
	\vspace{-10pt}
	\caption{(a) and (b): $\mathbb{E}[\tau | \tau < T]$; (c) and (d) $\mathbb{V}[\tau | \tau < T]$ for different values of $\alpha$ and $T$.}
 	\label{fig_moments}
\end{figure}

\section{Conclusion}
\label{sec:conclusion}
We have developed a semi-analytical approach to finding the density of interacting particles where their common downward drift increases in magnitude when particles hit a lower boundary, thus creating a positive feedback effect. This leads to a nonlinear and nonlocal parabolic equation. Using the method of heat potentials, we derived an equivalent coupled system of Volterra integral equations and solved it numerically, or by expansion for a small interaction parameter $\alpha$. We confirmed empirically the convergence of order $1$ of the numerical method and demonstrated its better complexity in comparison to the particle method in \cite{kaushansky2018simulation}.
There are striking financial implications as the computations uncover, in a very simplified setting, how mutual liabilities accelerate defaults of individual banks.

This paper raises several open questions. The numerical method for the system of Volterra equations can be improved using the methods we described at the beginning of Section \ref{numerical_solution}; one can potentially analyse the convergence of the method for the blow-up case. Another interesting direction is to study a model with common noise as in \cite{hambly2018spde} and \cite{ledger2018mercy} using the method of heat potentials. Lastly, it would be interesting to investigate an extension of the current paper for more complicated diffusion equations such as those in \cite{Shkolnikov} and \cite{carrillo2013classical}.
\bigskip
\bigskip

{\bf Acknowledgements:} We thank Ben Hambly and Andreas Sojmark for discussions on the theoretical properties of their model and
its link to the Stefan problem.

\begin{appendices}

\section{Derivation of limits in Section \ref{sec:HP}}
\label{app:limits}

We start with 
(\ref{Eq11a}).
We split $\mathbb{L}_{1}$ into two
parts, %
\begin{equation*}
\mathbb{L}_{1}=\mathbb{L}_{1}^{\left( 1\right) }-\mathbb{L}_{1}^{\left(
2\right) }, 
\end{equation*}%
where%
\begin{equation*}
\begin{aligned}
\mathbb{L}_{1}^{\left( 1\right) }&=\underset{x\rightarrow 0}{\lim }\left(
x\int_{0}^{t}\frac{\exp \left( -\frac{\left( x-\Psi \left( t,t^{\prime
}\right) \right) ^{2}}{2\left( t-t^{\prime }\right) }\right) }{\sqrt{2\pi
\left( t-t^{\prime }\right) ^{3}}}\nu \left( t^{\prime }\right) dt^{\prime
}\right) , \\ 
\mathbb{L}_{1}^{\left( 2\right) }&=\underset{x\rightarrow 0}{\lim }%
\int_{0}^{t}\frac{\Psi \left( t,t^{\prime }\right) \exp \left( -\frac{\left(
x-\Psi \left( t,t^{\prime }\right) \right) ^{2}}{2\left( t-t^{\prime
}\right) }\right) }{\sqrt{2\pi \left( t-t^{\prime }\right) ^{3}}}\nu \left(
t^{\prime }\right) dt^{\prime }.%
\end{aligned}
\end{equation*}%
We represent $\nu \left( t^{\prime }\right) =\nu \left( t\right) +\left( \nu
\left( t^{\prime }\right) -\nu \left( t\right) \right) $, and write%
\begin{equation*}
\begin{aligned}
\mathbb{L}_{1}^{\left( 1\right) }&=\underset{x\rightarrow 0}{\lim }\left(
x\nu \left( t\right) \int_{0}^{t}\frac{\exp \left( -\frac{\left( x-\Psi
\left( t,t^{\prime }\right) \right) ^{2}}{2\left( t-t^{\prime }\right) }%
\right) }{\sqrt{2\pi \left( t-t^{\prime }\right) ^{3}}}\, dt^{\prime }\right)
\\ 
&+\underset{x\rightarrow 0}{\lim }\left( x\int_{0}^{t}\frac{\exp \left( -%
\frac{\left( x-\Psi \left( t,t^{\prime }\right) \right) ^{2}}{2\left(
t-t^{\prime }\right) }\right) }{\sqrt{2\pi \left( t-t^{\prime }\right) ^{3}}}%
\left( \nu \left( t^{\prime }\right) -\nu \left( t\right) \right) \, dt^{\prime
}\right) \\ 
&=\nu \left( t\right) \underset{x\rightarrow 0}{\lim }\left( x\int_{0}^{t}%
\frac{\exp \left( -\frac{\left( x-\Psi \left( t,t^{\prime }\right) \right)
^{2}}{2\left( t-t^{\prime }\right) }\right) }{\sqrt{2\pi \left( t-t^{\prime
}\right) ^{3}}}\, dt^{\prime }\right) ,%
\end{aligned}
\end{equation*}%
since the second integral converges. We use the change of variables 
$t-t^{\prime }=x^{2}u$,  
so that %
\[
\begin{aligned}
\underset{x\rightarrow 0}{\lim }\left( x\int_{0}^{t}\frac{\exp \left( -\frac{%
\left( x-\Psi \left( t,t^{\prime }\right) \right) ^{2}}{2\left( t-t^{\prime
}\right) }\right) }{\sqrt{2\pi \left( t-t^{\prime }\right) ^{3}}}\, dt^{\prime
}\right) 
=\underset{x\rightarrow 0}{\lim }\int_{0}^{t/x^{2}}\frac{\exp \left( -\frac{1%
}{2u}\right) }{\sqrt{2\pi u^{3}}} \, du =\int_{0}^{\infty }\frac{\exp \left( -\frac{1}{2u}\right) }{\sqrt{2\pi u^{3}}%
} \, du =1.%
\end{aligned}%
\]
Since the second integral converges, it is clear that%
\begin{equation*}
\mathbb{L}_{1}^{\left( 2\right) }=\int_{0}^{t}\frac{\Psi \left( t,t^{\prime
}\right) \Xi \left( t,t^{\prime }\right) }{\sqrt{2\pi \left( t-t^{\prime
}\right) ^{3}}}\nu \left( t^{\prime }\right) dt^{\prime }.  \label{Eq44}
\end{equation*}%
Thus, (\ref{Eq11a}) is valid.

The second limit $\mathbb{L}_{2}$ in (\ref{Eq11d}) is less standard and more difficult to
evaluate. 
We introduce a non-singular function $\phi $,%
\begin{equation*}
\begin{aligned}
\phi \left( t,t^{\prime }\right) &=\frac{\Psi \left( t,t^{\prime }\right) }{%
t-t^{\prime }}, \quad t \neq t', \qquad\quad
\phi \left( t,t\right) =M_{t}\left( t\right), %
\end{aligned}
\end{equation*}%
and write%
\begin{equation*}
\begin{aligned}
\mathbb{{L}}_{2}
&=\underset{x\rightarrow 0}{\lim }\int_{0}^{t}\left( 1-\frac{x^{2}}{\left(
t-t^{\prime }\right) }+2x\phi \left( t,t^{\prime }\right) -\phi \left(
t,t^{\prime }\right) ^{2}\left( t-t^{\prime }\right) \right) \\ 
&\times \frac{\exp \left( -\frac{x^{2}}{2\left( t-t^{\prime }\right) }+x\phi
\left( t,t^{\prime }\right) \right) \Xi \left( t,t^{\prime }\right) }{\sqrt{%
2\pi \left( t-t^{\prime }\right) ^{3}}}\nu \left( t^{\prime }\right)
\, dt^{\prime } \\ 
&=\mathbb{{L}}_{2}^{\left( 1\right) }+\mathbb{{L}}_{2}^{\left(
2\right) }+2\mathbb{{L}}_{2}^{\left( 3\right) }-\mathbb{{L}}%
_{2}^{\left( 4\right) },%
\end{aligned}
\end{equation*}%
where%
\begin{equation*}
\begin{aligned}
\mathbb{{L}}_{2}^{\left( 1\right) }&=\underset{x\rightarrow 0}{\lim }%
\, \nu \left( t\right) \int_{0}^{t}\left( 1-\frac{x^{2}}{\left( t-t^{\prime
}\right) }\right) \frac{\exp \left( -\frac{x^{2}}{2\left( t-t^{\prime
}\right) }+2x\phi \left( t,t^{\prime }\right) \right) }{\sqrt{2\pi \left(
t-t^{\prime }\right) ^{3}}} \, dt^{\prime }, \\ 
\mathbb{{L}}_{2}^{\left( 2\right) }&=\underset{x\rightarrow 0}{\lim }%
\int_{0}^{t}\left( 1-\frac{x^{2}}{\left( t-t^{\prime }\right) }\right) \frac{%
\exp \left( -\frac{x^{2}}{2\left( t-t^{\prime }\right) }+2x\phi \left(
t,t^{\prime }\right) \right) }{\sqrt{2\pi \left( t-t^{\prime }\right) ^{3}}}%
\left( \Xi \left( t,t^{\prime }\right) \nu \left( t^{\prime }\right) -\nu
\left( t\right) \right) \, dt^{\prime }, \\ 
\mathbb{{L}}_{2}^{\left( 3\right) }&=\underset{x\rightarrow 0}{\lim }%
\, x\int_{0}^{t}\phi \left( t,t^{\prime }\right) \frac{\exp \left( -\frac{x^{2}%
}{2\left( t-t^{\prime }\right) }+2x\phi \left( t,t^{\prime }\right) \right)
\Xi \left( t,t^{\prime }\right) }{\sqrt{2\pi \left( t-t^{\prime }\right) ^{3}%
}}\nu \left( t^{\prime }\right) \, dt^{\prime }, \\ 
\mathbb{{L}}_{2}^{\left( 4\right) }&=\underset{x\rightarrow 0}{\lim }%
\int_{0}^{t}\phi \left( t,t^{\prime }\right) ^{2}\frac{\exp \left( -\frac{%
x^{2}}{2\left( t-t^{\prime }\right) }+x\phi \left( t,t^{\prime }\right)
\right) \Xi \left( t,t^{\prime }\right) }{\sqrt{2\pi \left( t-t^{\prime
}\right) }}\nu \left( t^{\prime }\right) \, dt^{\prime }.%
\end{aligned}
\end{equation*}%
We have%
\begin{equation*}
\begin{aligned}
\mathbb{{L}}_{2}^{\left( 1\right) }&=\underset{x\rightarrow 0}{\lim }%
\, \nu \left( t\right) \int_{0}^{t}\left( 1-\frac{x^{2}}{\left( t-t^{\prime
}\right) }\right) \frac{\exp \left( -\frac{x^{2}}{2\left( t-t^{\prime
}\right) }+2x\phi \left( t,t^{\prime }\right) \right) }{\sqrt{2\pi \left(
t-t^{\prime }\right) ^{3}}} \, dt^{\prime } \\ 
&=\nu \left( t\right) \underset{x\rightarrow 0}{\lim }\frac{1}{x}%
\int_{0}^{t/x^{2}}\left( 1-\frac{1}{u}\right) \frac{\exp \left( -\frac{1}{2u}%
\right) }{\sqrt{2\pi u^{3}}}\, du \\ 
&=2\nu \left( t\right) \underset{x\rightarrow 0}{\lim }\frac{1}{x}\int_{x/%
\sqrt{t}}^{\infty }\left( 1-v^{2}\right) \frac{\exp \left( -\frac{v^{2}}{2}%
\right) }{\sqrt{2\pi }} \, dv \\ 
&=-2\nu \left( t\right) \underset{x\rightarrow 0}{\lim }\frac{1}{x}%
\int_{0}^{x/\sqrt{t}}\left( 1-v^{2}\right) \frac{\exp \left( -\frac{v^{2}}{2}%
\right) }{\sqrt{2\pi }} \, dv \\ 
&=-\frac{2}{\sqrt{2\pi t}}\nu \left( t\right),%
\end{aligned}
\end{equation*}%
where $\left( t-t^{\prime }\right) =u$, $u=1/v^{2}$ and we have used the fact
that%
\begin{equation*}
\int_{0}^{\infty }\left( 1-v^{2}\right) \frac{\exp \left( -\frac{v^{2}}{2}%
\right) }{\sqrt{2\pi }} \, dv=0.%
\end{equation*}
Further,%
\begin{equation*}
\begin{aligned}
\mathbb{{L}}_{2}^{\left( 2\right) }&=\underset{x\rightarrow 0}{\lim }%
\int_{0}^{t}\left( 1-\frac{x^{2}}{\left( t-t^{\prime }\right) }\right) \frac{%
\exp \left( -\frac{x^{2}}{2\left( t-t^{\prime }\right) }+2x\phi \left(
t,t^{\prime }\right) \right) }{\sqrt{2\pi \left( t-t^{\prime }\right) ^{3}}}%
\left( \Xi \left( t,t^{\prime }\right) \nu \left( t^{\prime }\right) -\nu
\left( t\right) \right) \, dt^{\prime } \\ 
&=\underset{x\rightarrow 0}{\lim }\int_{0}^{t}\frac{\exp \left( -\frac{x^{2}}{%
2\left( t-t^{\prime }\right) }+2x\phi \left( t,t^{\prime }\right) \right) }{%
\sqrt{2\pi \left( t-t^{\prime }\right) ^{3}}}\left( \Xi \left( t,t^{\prime
}\right) \nu \left( t^{\prime }\right) -\nu \left( t\right) \right) \,
dt^{\prime } \\ 
&=\int_{0}^{t}\frac{\left( \Xi \left( t,t^{\prime }\right) \nu \left(
t^{\prime }\right) -\nu \left( t\right) \right) }{\sqrt{2\pi \left(
t-t^{\prime }\right) ^{3}}} \, dt^{\prime },%
\end{aligned}
\end{equation*}%
where we have dropped the higher order $x^2$ term in the integral 
in the second line,
\begin{equation*}
\begin{aligned}
\mathbb{{L}}_{2}^{\left( 3\right) }&=\underset{x\rightarrow 0}{\lim }%
x\int_{0}^{t}\phi \left( t,t^{\prime }\right) \frac{\exp \left( -\frac{x^{2}%
}{2\left( t-t^{\prime }\right) }+2x\phi \left( t,t^{\prime }\right) \right) 
}{\sqrt{2\pi \left( t-t^{\prime }\right) ^{3}}}\Xi \left( t,t^{\prime
}\right) \nu \left( t^{\prime }\right) \, dt^{\prime } \\ 
&=\phi \left( t,t\right) \Xi \left( t,t\right) \nu \left( t\right) =M_{t}\left( t\right) \nu \left( t\right) ,%
\end{aligned}
\end{equation*}%
and
\begin{equation*}
\begin{aligned}
\mathbb{{L}}_{2}^{\left( 4\right) }&=\underset{x\rightarrow 0}{\lim }%
\int_{0}^{t}\phi \left( t,t^{\prime }\right) ^{2}\frac{\exp \left( -\frac{%
x^{2}}{2\left( t-t^{\prime }\right) }+x\phi \left( t,t^{\prime }\right)
\right) }{\sqrt{2\pi \left( t-t^{\prime }\right) }}\Xi \left( t,t^{\prime
}\right) \nu \left( t^{\prime }\right) \, dt^{\prime } \\ 
&=\int_{0}^{t}\phi \left( t,t^{\prime }\right) ^{2}\frac{\Xi \left(
t,t^{\prime }\right) }{\sqrt{2\pi \left( t-t^{\prime }\right) }}\nu \left(
t^{\prime }\right) \, dt^{\prime } =\int_{0}^{t}\frac{\Psi \left( t,t^{\prime }\right) ^{2}}{\left( t-t^{\prime
}\right) }\frac{\Xi \left( t,t^{\prime }\right) }{\sqrt{2\pi \left(
t-t^{\prime }\right) ^{3}}}\nu \left( t^{\prime }\right) \, dt^{\prime }.%
\end{aligned}
\end{equation*}

Finally,%
\begin{equation*}
\begin{aligned}
\mathbb{{L}}_{2}&=\mathbb{{L}}_{2}^{\left( 1\right) }+\mathbb{%
{L}}_{2}^{\left( 2\right) }+2\mathbb{{L}}_{2}^{\left( 3\right) }-%
\mathbb{{L}}_{2}^{\left( 4\right) } \\ 
&=-2\frac{1}{\sqrt{2\pi t}}\nu \left( t\right) +\int_{0}^{t}\frac{\left( \Xi
\left( t,t^{\prime }\right) \nu \left( t^{\prime }\right) -\nu \left(
t\right) \right) }{\sqrt{2\pi \left( t-t^{\prime }\right) ^{3}}} \, dt^{\prime }
\\ 
&+2M_{t}\left( t\right) \nu \left( t\right) -\int_{0}^{t}\frac{\Psi \left(
t,t^{\prime }\right) ^{2}}{\left( t-t^{\prime }\right) }\frac{\Xi \left(
t,t^{\prime }\right) }{\sqrt{2\pi \left( t-t^{\prime }\right) ^{3}}}\nu
\left( t^{\prime }\right) \, dt^{\prime } \\ 
&=2\left( M_{t}\left( t\right) -\frac{1}{\sqrt{2\pi t}}\right) \nu \left(
t\right) +\int_{0}^{t}\frac{\left( \left( 1-\frac{\Psi \left( t,t^{\prime
}\right) ^{2}}{\left( t-t^{\prime }\right) }\right) \Xi \left( t,t^{\prime
}\right) \nu \left( t^{\prime }\right) -\nu \left( t\right) \right) }{\sqrt{%
2\pi \left( t-t^{\prime }\right) ^{3}}} \, dt^{\prime },%
\end{aligned}
\end{equation*}%
as stated.

\section{Special cases}
\label{app_special_cases}

For illustration, we work out the solution from the formula obtained in Section \ref{sec:HP} for two special cases which are also accessible by
the standard reflection principle for Brownian motion or method of images for parabolic equations.

\subsection{$M(t) = 0$}
When $M\left( t\right) =0$, we get%
\begin{equation*}
\left\{ 
\begin{aligned}
& \nu \left( t\right) +\frac{\exp \left( -\frac{z^{2}}{2t}\right) }{\sqrt{2\pi
t}}=0, \\ 
& g\left( t\right) +\frac{1}{\sqrt{2\pi t}}\nu \left( t\right) +\frac{1}{2}%
\int_{0}^{t}\frac{\left( \nu \left( t\right) -\nu \left( t^{\prime }\right)
\right) }{\sqrt{2\pi \left( t-t^{\prime }\right) ^{3}}} \, dt^{\prime }-\frac{%
z\exp \left( -\frac{z^{2}}{2t}\right) }{2\sqrt{2\pi t^{3}}}=0.%
\end{aligned}%
\right.  
\end{equation*}%
Integration by parts of the second equation and use of the first yields 
\begin{equation*}
\begin{aligned}
g\left( t\right)&=-\int_{0}^{t}\frac{\overset{\cdot }{\nu }\left( t^{\prime
}\right) }{\sqrt{2\pi \left( t-t^{\prime }\right) }}\, dt^{\prime }+\frac{z\exp
\left( -\frac{z^{2}}{2t}\right) }{2\sqrt{2\pi t^{3}}} \\ 
&=\frac{z\exp \left( -\frac{z^{2}}{2t}\right) }{2\sqrt{2\pi t^{3}}}+\frac{%
z\exp \left( -\frac{z^{2}}{2t}\right) }{2\sqrt{2\pi t^{3}}} \\ 
&=\frac{z\exp \left( -\frac{z^{2}}{2t}\right) }{\sqrt{2\pi t^{3}}},%
\end{aligned}
\label{Eq15a}
\end{equation*}%
as expected.
\subsection{$M(t) = \mu t$}
When $M\left( t\right) =\mu t$, we get
\begin{equation}
\left\{
\begin{aligned}
	&\nu(t) - \mu \int_0^t \frac{\exp \left( - \frac{\mu^2(t - t')}{2} \right)}{\sqrt{2 \pi (t - t')}} \nu(t') d t' + \frac{\exp \left( - \frac{(\mu t + z)^2}{2 t} \right)}{\sqrt{2 \pi t}} = 0, \\
	&g\left( t\right) = \frac{1}{2 \sqrt{2 \pi}} \int_{0}^{t}   \frac{1}{\sqrt{(t - t')^3}} \left( \exp \left( - \frac{\mu^2(t - t')}{2} \right)\nu \left( t^{\prime }\right) - \nu(t) \right) d t' \\
	 &-\frac{\mu^2}{2\sqrt{2 \pi}} \int_{0}^{t}  \frac{ 1}{\sqrt{t - t'}} \exp \left( - \frac{\mu^2(t - t')}{2} \right)  \nu \left( t^{\prime }\right) d t' + \left( \mu - \frac{1}{\sqrt{2 \pi t}} \right) \nu(t) \\
	&+\frac{\left( \mu t +z\right) 
	\exp \left( -\frac{\left( \mu t +z\right) ^{2}%
	}{2t}\right) }{\sqrt{2\pi t^3}}.
\end{aligned}
\right.
\label{eq_mt}
\end{equation}
Taking the Laplace transform of the first equation in \eqref{eq_mt}, we have
\begin{equation*}
	\hat{\nu}(s) - \mu \hat{\nu}(s) \frac{1}{\sqrt{2 s + \mu^2}} + \frac{\exp(-z \sqrt{2 s + \mu^2} - \mu z) }{\sqrt{2 s + \mu^2}} = 0.
\end{equation*}
Hence,
\begin{equation*}
	\hat{\nu}(s) = -  \frac{\exp(-z \sqrt{2 s + \mu^2} - \mu z) }{\sqrt{2 s + \mu^2} - \mu}.
\end{equation*}
The inverse Laplace transform of $\hat{\nu}(s)$ can be found analytically, but we do not need it to compute $g(t)$. Consider the second equation of \eqref{eq_mt}. The first integral can be rewritten as
\begin{equation*}
\begin{aligned}
	&\frac{1}{2 \sqrt{2 \pi}} \int_{0}^{t}   \frac{1}{\sqrt{(t - t')^3}} \left( \exp \left( - \frac{\mu^2(t - t')}{2} \right)\nu \left( t^{\prime }\right) - \nu(t) \right) d t' \\ 
	&= \frac{1}{ \sqrt{2 \pi}} \int_{0}^{t}  \left(\exp \left( - \frac{\mu^2(t - t')}{2} \right)\nu \left( t^{\prime }\right)  - \nu(t) \right) d \left( \frac{1}{\sqrt{(t - t')}} \right) \\
	&= -\frac{1}{ \sqrt{2 \pi t}}   \left(\exp \left( - \frac{\mu^2t}{2} \right)\nu \left( 0\right)  - \nu(t) \right) - \frac{1}{ \sqrt{2 \pi}} \int_{0}^{t} \frac{1}{\sqrt{t - t'}}   \exp \left( - \frac{\mu^2(t - t')}{2} \right)\ \dot{\nu} \left( t^{\prime }\right) d t' \\
	&-\frac{\mu^2}{ 2\sqrt{2 \pi}} \int_{0}^{t}  \frac{1}{\sqrt{(t - t')}}   \exp \left( - \frac{\mu^2(t - t')}{2} \right)\nu \left( t^{\prime }\right) d t'.
\end{aligned}
\end{equation*}
Thus,
\begin{equation*}
\begin{aligned}
	g(t) &= \mu \nu(t) -\frac{1}{ \sqrt{2 \pi t}} \exp \left( - \frac{\mu^2t}{2} \right)\nu \left( 0\right) - \frac{1}{ \sqrt{2 \pi}} \int_{0}^{t} \frac{1}{\sqrt{t - t'}}   \exp \left( - \frac{\mu^2(t - t')}{2} \right)\ \dot{\nu} \left( t^{\prime }\right) d t' \\
	&-\frac{\mu^2}{ \sqrt{2 \pi}} \int_{0}^{t}  \frac{1}{\sqrt{(t - t')}}   \exp \left( - \frac{\mu^2(t - t')}{2} \right)\nu \left( t^{\prime }\right) d t'
+\frac{\left( \mu t +z\right) 
\exp \left( -\frac{\left( \mu t +z\right) ^{2}%
}{2t}\right) }{2 \sqrt{2\pi t^3}}.
\end{aligned}
\end{equation*}
Taking Laplace transform of the last equation, we get
\begin{equation*}
\begin{aligned}
	\hat{g}(s) &= \mu \hat{\nu}(s) - \frac{1}{\sqrt{2s + \mu^2}} \nu(0) - \frac{1}{\sqrt{2s + \mu^2}}  \left( s \hat{\nu}(s) - \nu(0) \right) - \frac{\mu^2}{\sqrt{2s + \mu^2}} \hat{\nu}(s) \\
	&+ \left( \frac{1}{2} + \frac{\mu}{ 2\sqrt{2s + \mu^2}} \right) \exp(-z \sqrt{2 s + \mu^2} + \mu z) \\
	&= \left(-\frac{\mu}{\sqrt{2s + \mu^2} - \mu} + \frac{s + \mu^2}{\sqrt{2s + \mu^2}} \frac{1 }{\sqrt{2 s + \mu^2} - \mu^2} \right.\\
	& \left. + \frac{1}{2} +  \frac{\mu}{2 \sqrt{2s + \mu^2}} \right) \exp(-z \sqrt{2 s + \mu^2} + \mu z) =  \exp(-z \sqrt{2 s + \mu^2} + \mu z).
	\end{aligned}
\end{equation*}
The inverse Laplace transform yields the final expression for $g(t)$
\begin{equation*}
	g(t) = \frac{z \exp \left( - \frac{(z - \mu t)^2}{2 t} \right)}{\sqrt{2 \pi t^3}},
\end{equation*}
as expected.

\section{Proof of Lemma \ref{lemma1}}
\label{app:prooflemma}

\begin{proof}
[Proof of Lemma \ref{lemma1}]
We start with the first term and judiciously use integration by parts several times to get
\begin{equation*}
\begin{aligned}
&\frac{d}{dt}\int_{0}^{t}\frac{\Xi \left( t,t^{\prime }\right) \nu \left(
t^{\prime }\right) }{\sqrt{2\pi \left( t-t^{\prime }\right) }} \, dt^{\prime } \\
&=-2\frac{d}{dt}\int_{0}^{t}\Xi \left( t,t^{\prime }\right) \nu \left(
t^{\prime }\right) d\sqrt{\frac{\left( t-t^{\prime }\right) }{2\pi }} \\
&=2\frac{d}{dt}\left( \Xi \left( t,0\right) \nu \left( 0\right) \sqrt{\frac{%
t}{2\pi }}+\int_{0}^{t}\sqrt{\frac{\left( t-t^{\prime }\right) }{2\pi }}%
d\left( \Xi \left( t,t^{\prime }\right) \nu \left( t^{\prime }\right)
\right) \right)  \\
&=\frac{\left( 2t\Xi _{t}\left( t,0\right) +\Xi \left( t,0\right) \right)
\nu \left( 0\right) }{\sqrt{2\pi t}}+\int_{0}^{t}\frac{1}{\sqrt{2\pi \left(
t-t^{\prime }\right) }}d\left( \Xi \left( t,t^{\prime }\right) \nu \left(
t^{\prime }\right) -\nu \left( t\right) \right)  \\
&+2\int_{0}^{t}\sqrt{\frac{\left( t-t^{\prime }\right) }{2\pi }}d\left( \Xi
_{t}\left( t,t^{\prime }\right) \nu \left( t^{\prime }\right) \right)  \\
&=\frac{\left( 2t\Xi _{t}\left( t,0\right) +\Xi \left( t,0\right) \right)
\nu \left( 0\right) }{\sqrt{2\pi t}}-\frac{\Xi \left( t,0\right) \nu \left(
0\right) -\nu \left( t\right) }{\sqrt{2\pi t}}-\frac{1}{2}\int_{0}^{t}\frac{%
\Xi \left( t,t^{\prime }\right) \nu \left( t^{\prime }\right) -\nu \left(
t\right) }{\sqrt{2\pi \left( t-t^{\prime }\right) ^{3}}}\, dt^{\prime } \\
&-2\Xi _{t}\left( t,0\right) \nu \left( 0\right) \sqrt{\frac{t}{2\pi }}%
+\int_{0}^{t}\frac{\Xi _{t}\left( t,t^{\prime }\right) }{\sqrt{2\pi \left(
t-t^{\prime }\right) }}\nu \left( t^{\prime }\right) \, dt^{\prime } \\
&=\frac{\nu \left( t\right) }{\sqrt{2\pi t}}+\frac{1}{2}\int_{0}^{t}\frac{%
\nu \left( t\right) -\left( \Xi \left( t,t^{\prime }\right) -2\left(
t-t^{\prime }\right) \Xi _{t}\left( t,t^{\prime }\right) \right) \nu \left(
t^{\prime }\right) }{\sqrt{2\pi \left( t-t^{\prime }\right) ^{3}}} \, dt^{\prime
} \\
&=\frac{\nu \left( t\right) }{\sqrt{2\pi t}}+\int_{0}^{t}\left( \nu \left(
t\right) -\left( \Xi \left( t,t^{\prime }\right) -2\left( t-t^{\prime
}\right) \Xi _{t}\left( t,t^{\prime }\right) \right) \nu \left( t^{\prime
}\right) \right) d\left( \frac{1}{\sqrt{2\pi \left( t-t^{\prime }\right) }}%
\right)  \\
&=\int_{0}^{t}\frac{\left( \left( \Xi \left( t,t^{\prime }\right) -2\left(
t-t^{\prime }\right) \Xi _{t}\left( t,t^{\prime }\right) \right) \nu \left(
t^{\prime }\right) \right) _{t^{\prime }}}{\sqrt{2\pi \left( t-t^{\prime
}\right) }}\, dt^{\prime }.
\end{aligned}
\end{equation*}
\end{proof}

\end{appendices}

\bibliographystyle{apalike}
\bibliography{main} 

\end{document}